\documentclass[11pt]{amsart}
\usepackage{amsfonts,amsmath,amssymb,amsthm}% Required for inserting images
\theoremstyle{plain}
\usepackage[nomarginpar, margin=2cm, foot=.25in]{geometry}
\usepackage{mathrsfs}
\usepackage{graphicx} 
\usepackage{mathrsfs}
\usepackage{fdsymbol}
\usepackage{hyperref}
\usepackage{enumitem}
\usepackage{longtable}
\usepackage{tikz-cd}
\usepackage{tabularx}
\newtheorem{ex}{Example}[section]

\newtheorem{thm}{Theorem}[section]
\newtheorem{dfn}{Definition}[section]
\newtheorem{cor}{Corollary}[section]
\newtheorem{lem}{Lemma}[section]
\newtheorem{prop}{Proposition}[section]
\newtheorem{rem}{Remark}[section]
\newtheorem*{thm-non}{Theorem}
\newtheorem*{cor-non}{Corollary}

\newcommand{\Spec}{\rm Spec}
\newcommand{\Out}{\rm Out}
\newcommand{\F}{\widehat{F}}
\newcommand{\bN}{\mathbb{N}}
\newcommand{\bZ}{\mathbb{Z}}

\newcommand{\GT}{\widehat{GT}}
\newcommand{\Gal}{\operatorname{Gal}_{\Q}}
\newcommand{\Q}{\mathbb{Q}}

\newcommand{\cC}{\mathcal {C}}

\newcommand\myeq{\mathrel{\overset{\makebox[0pt]{\mbox{\normalfont\tiny\sffamily homeo}}}{\simeq}}}

\title[Grothendieck's conjecture $\&$ Galois's Path Integral]{Proving the Grothendieck--Teichmüller Conjecture for Profinite Spaces \\ $\&$ \\ The Galois Grothendieck Path Integral}
\author{Noémie C. Combe}
\address{University of Warsaw, Ulica Banacha 2, 02-097 Warsaw}\email{ n.combe@uw.edu.pl}
\date{}

\begin{document}

\begin{abstract}
We establish that the Grothendieck–Teichmüller conjecture, which predicts an isomorphism between the Grothendieck–Teichmüller group $\GT$ and the absolute Galois group $\Gal$ of $\mathbb{Q}$, holds in the setting of profinite spaces. To access arithmetic information within this framework, we introduce a generalization of the notion of path integrals, defined simultaneously for $\GT$ and $\Gal$. This construction reveals new arithmetic invariants for both groups, capturing, in particular, the rationality of periods associated with $\GT$ and the rational structure of the coefficients in the Drinfeld associator. Moreover, this perspective provides a mechanism to detect the descent of cohomology classes to $\mathbb{Q}$ under the action of $\Gal$, and more broadly, a method to compute ranks of arithmetic objects. Furthermore, we introduce an algorithm, which we call the Cubic Matrioshka, encoding paths in both $\GT$ and $\Gal$ as unique infinite binary sequences. The compatibility of these binary encodings reflects the underlying homeomorphism and ensures that symmetries of $\GT$ are faithfully mirrored in the corresponding binary representation within $\Gal$.

\end{abstract}
\thanks{{\bf Acknowledgements.} I am deeply grateful, first and foremost, to the referee, whose insightful and generous suggestion regarding the compatibility of Cantorian paths in $\Gal$ and $\GT$ played a pivotal role in shaping theorem~\ref{T:Compatible}.  I also warmly thank Pierre Lochak for pointing out a small inaccuracy in the bibliography.
I sincerely thank Vasily Dolgushev for enlightening discussions on  $\GT$-shadows, which have deeply influenced my understanding of their structure and inspired new directions in this work. I would like to thank John S. Wilson, for having given a lecture on profinite groups, which I attended as an undergraduate student, and that has provided the necessary insights for the developments presented in this work. I would like to express my gratitude to Rostislav Grigorchuk for his inspiring lecture on finitely generated groups, which has inspired me the Cubic Matrioshka Algorithm construction. I am grateful for having been invited to Oberwolfach MFO, on {\it Arithmetic and Homotopic Galois Theory} Sept. 2023 (ID: 2339a)
as well as  {\it Combinatorial and Algebraic Structures in Rough Analysis and Related Fields} Nov. 2023 (ID: 2348a),
 where I had the opportunity to gain valuable insights into the state of the art in the fields and that inspired me this paper. This research is part of the project No. 2022/47/P/ST1/01177 cofunded by the National Science Centre  and the European Union's Horizon 2020 research and innovation program, under the Marie Sklodowska Curie grant agreement No. 945339 \includegraphics[width=1cm, height=0.5cm]{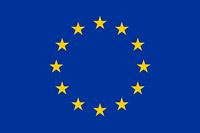}. For the
purpose of Open Access, the author has applied a CC-BY public copyright license
to any Author Accepted Manuscript (AAM) version arising from this submission.}

\maketitle


\, 

Secondary: 57M07; 54G12; 81S40

\smallskip 


\section{Introduction}
%SA
The absolute Galois group $\Gal=\operatorname{Gal}(\overline{\Q}/\Q)$ of rational numbers stands as one of the most fundamental and enigmatic objects in modern algebraic geometry and number theory. Explicitly describing the elements of $\Gal$ is a profound open problem. Numerous conjectures aim at characterising $\Gal$ (or its representations), such as the Inverse Galois Problem, Grothendieck’s Section Conjecture, Shafarevich’s Dream, and the Grothendieck--Teichmüller Conjecture. A deep understanding of its topological, combinatorial, and algebraic structure holds far-reaching implications for anabelian geometry, the study of Diophantine equations, and the Langlands program \cite{Neu,NSW,Len,W,Ser,Ser2,Ser3}.

%Challenge

\subsection{The Grothendieck--Teichmüller Conjecture}

\subsubsection{Introducing $\GT$}
Parallel to this, the Grothendieck–Teichmüller group $\GT$ plays a fundamental role in arithmetic geometry, anabelian geometry, and the study of profinite fundamental groups of algebraic curves \cite{CM1,CM2,CMM1,CMM2,M,HS,PS,LS,Sch}. It serves as a universal symmetry group encoding the hidden Galois-theoretic structures underlying moduli spaces of algebraic curves and their fundamental groups as well as in deformation theory. 

The group $\GT$ consists of pairs $(\lambda,f)$, where $\lambda \in \widehat{\mathbb{Z}}^{\times}$ (the profinite completion of $\mathbb{Z}^{\times}$) and $f\in \widehat{F}_2$  the profinite completion of the free group on two generators, say $x,y$ satisfying relations:
\begin{itemize} 
\item $f(x,y)f(y,x)=1$,
\item $f(z,x)z^mf(y,z)y^mf(x,y)x^m=1$ for $m=(\lambda-1)/2$ and  $z=(xy)^{-1}$,
\item $f(x_{12},x_{23}x_{24})f(x_{13}x_{23},x_{34})=f(x_{23},x_{34})f(x_{12}x_{13},x_{24}x_{34})f(x_{12},x_{23})$.
\end{itemize}
The last equation takes place in the profinite completion of the pure braid group with four strands $\widehat{PB}_4$, with generators $x_{ij}$. 
\smallskip 

Despite their distinct origins, the structural properties of
$\Gal$ and $\GT$ exhibit profound similarities, a phenomenon at the heart of Grothendieck’s vision of Galois--Teichmüller theory.

\subsubsection{Grothendieck's Enigma}

The longstanding Grothendieck--Teichmüller Conjecture asserts the existence of an isomorphism of profinite groups between 
$\GT$ and $\Gal$ i.e. 
\[\Gal\cong \GT.\] Drinfel'd \cite{Dr} established that there exists a natural embedding 
$\Gal\hookrightarrow \GT$, yet the question of whether this morphism is surjective, and hence an isomorphism, remains open.

\medskip 

 \subsection{Result Description $\&$ Modus Operandi}
\begin {itemize}
\item In this work, we demonstrate that the Grothendieck--Teichmüller Conjecture holds in the category of profinite spaces {\bf ProFin} (Thm. \ref{T:Cantor}): 
\[\Gal\quad \myeq \quad \widehat{GT}\quad  \myeq\quad\text{Cantor set}.\]
 Our {\it modus operandi} is framed within the context of Polish group theory, wherein we establish that both $\Gal$ and $\GT$ are homeomorphic to the Cantor set (see Appendix Sec.~\ref{S:2}-\ref{S:3} and Sec.\ref{S:5}). \item  In Theorem \ref{T:GPI}, we generalise the notion of Feynman Path Integral to the arithmetic setting by introducing the notion of a Galois Grothendieck Path Integral, on $\Gal$ and respectively on $\GT$. The motivation is to recover the initial arithmetic data which got lost during the topological construction in the category {\bf ProFin}. This leads to the elaboration of arithmetic {\bf invariants} of the absolute Galois and the Grothendieck--Teichmüller groups (Theorem~\ref{T:9.2} and Theorem~\ref{T:9.3}). Both invariants are path integrals, where trajectories are automorphisms in $\Gal$ (resp. $\GT$). 
These integrals are defined by:
 \[
I([c]) = \int_{\Gal} \exp(2\pi i \langle \sigma, [c] \rangle) \, d\mu(\sigma), \quad
I([\Psi]) = \int_{\widehat{GT}} \exp(2\pi i \langle \theta, [\Psi] \rangle) \, d\mu(\theta),
\] 

and such that 
\begin{enumerate}
\item  the {\bf measure} $\mu$ is the normalized Haar measure, 
\item  {\bf Paths} are elements $\sigma$ (res. $\theta$) forming {\bf arithmetic paths.} Each automorphism $\sigma\in \Gal$ (resp. $\theta\in \GT$) is interpreted as an arithmetic path, because this encodes sequences of field extensions, with automorphisms  transitioning between étale covers. %A similar idea holds for $\theta \in \GT$. 
\item {\bf Arithmetic Functional:}  $\langle \cdot, \cdot \rangle$ be a natural Galois  pairing for $\Gal$ (resp. motivic for $\GT$), explained below:
\begin{itemize}
\item Let $[c]$ be a {\bf cohomology class} i.e. an element in $H^1_{\mathrm{\acute{e}t}}(M,\mathbb{Q}/\mathbb{Z})$, where $M/\Q$ is a motive.
\item The {\bf pairing} $\langle \sigma, [c]\rangle$ is defined via Kummer duality or evaluation of cocycles (in more general settings):

\[\langle \sigma, [c] \rangle:\Gal\times H^1_{\mathrm{\acute{e}t}}(M)\longrightarrow \,  \mathbb{Q}/\mathbb{Z}.\] 
This is the evaluation of the Galois action on the class $[c]$, arising from the exact sequence:
\[0\to  \mathbb{Q}/\mathbb{Z} \to H^1_{\mathrm{\acute{e}t}}(M_{\overline{\Q}}, \Q/\bZ) \to H^1_{\mathrm{\acute{e}t}}(M_{\Q},\Q/\bZ)\to 0, \]
\,
which allows, via the exponential, to define a continuous finite-order character \[\xi_{[c]}:\Gal \to \mathbb{C}^{\times}, \]
given by \[\xi_{[c]}(\sigma):= \exp(2\pi \imath \langle \sigma,[c]\rangle),\] and where
$\xi_{[c]}$ takes values in roots of unity (since $\langle \sigma,[c]\rangle\in  \mathbb{Q}/\mathbb{Z}$). 
It is non-trivial unless $[c]$ is $\Gal$-invariant (that is $[c]$ descends to $\Q$).

\end{itemize}

\item For $\GT$ this invariant detects triviality in the category of mixed Tate motives. The pairing $\langle \theta, [\Psi] \rangle = \theta(\Psi_{\text{mot}}) \in \mathbb{Q}$ is defined via the coaction of the motivic Galois group.  Let $\mathcal{MT}(\mathbb{Z})$ be the category of mixed Tate motives over $\mathbb{Z}$, with motivic Galois group $G_{\mathcal{MT}(\mathbb{Z})}$. The coaction is a morphism:
\[
\Delta : \mathcal{H}_{\text{mot}} \to \mathcal{H}_{\text{mot}} \otimes \mathcal{H}_{\text{mot}},
\]
where $\mathcal{H}_{\text{mot}} = \mathcal{O}(G_{\mathcal{MT}(\mathbb{Z})})$ is the Hopf algebra of motivic periods. For a motivic period $[\Psi]$, represented by $\Psi_{\text{mot}} \in \mathcal{H}_{\text{mot}}$, the coaction decomposes it as:
\[
\Delta(\Psi_{\text{mot}}) = \sum_i \Psi_{\text{mot}}^{(i)} \otimes h_i, \quad h_i \in \mathcal{H}_{\text{mot}}.
\]

The Grothendieck-Teichmüller group $\widehat{GT}$ embeds into $G_{\mathcal{MT}(\mathbb{Z})}(\mathbb{Q})$ via 
\[
\iota: \widehat{GT} \hookrightarrow G_{\mathcal{MT}(\mathbb{Z})}(\mathbb{Q}).
\]
For $\theta \in \widehat{GT}$, the evaluation $\theta(\Psi_{\text{mot}})$ is defined by applying the coaction followed by evaluation at $\iota(\theta)$:
\[
\theta(\Psi_{\text{mot}}) = \sum_i \Psi_{\text{mot}}^{(i)} \cdot h_i(\iota(\theta)) \in \mathbb{Q}.
\]
This is equivalent to:
\[
\langle \theta, [\Psi] \rangle = (\operatorname{id} \otimes \operatorname{ev}_{\iota(\theta)}) \circ \Delta(\Psi_{\text{mot}}),
\]
where $\operatorname{ev}_{\iota(\theta)} : \mathcal{H}_{\text{mot}} \to \mathbb{Q}$ is the evaluation map at $\iota(\theta)$.
\end{enumerate}

 \item We close that section by  generalising this construction in order to compute ranks of arithmetic objects (e.g., Mordell-Weil rank). Precisely, if $\rho: \Gal \to \mathrm{GL}_n(\mathbb{Q}_\ell)$ is a continuous $\ell$-adic Galois representation, we define a generalization of the arithmetic path integral via:
\[
I(\rho) = \int_{\Gal} \mathrm{Tr}(\rho(\sigma)) \, d\mu(\sigma) = \dim_{\mathbb{Q}_\ell}(\rho^{\Gal}),
\]
where $\rho^{\Gal}$ is the subspace of $\Gal$-invariant vectors. We refer to Table~\ref{T:2} for a discussion concerning the invariants. 

\item[]
 \item In Sec.~\ref{S:6}--\ref{S:CubicMa} we discuss more deeply the construction of arithmetic paths. Every path on $\Gal$ (and of $\GT$) admits a unique representation as an infinite binary sequence (an infinite word of 0's and 1's). Section \ref{S:CubicMa} introduces the Cubic Matrioshka Algorithm (CMA): a concrete construction enabling an explicit characterisation of those paths  in terms of their profinite topology, applicable to both $\GT$ and $\Gal$. Our algorithm allows to keep track of paths of field extensions and corresponding automorphisms (see Proposition \ref{P:BG}).  Note however, that this homeomorphism is {\it not} canonical and therefore, requires to {\it fix} a convention of choices of clopens of $\Gal$ and $\GT$ in our construction. 

 \item These binary encodings in $\Gal$ and $\GT$ are {\it compatible} by Theorem \ref{T:Compatible} under the natural correspondence between $\Gal$ and $\GT$. In particular, the infinite binary sequence attached to any element of $\Gal$ uniquely determines the binary sequence of the corresponding element in $\GT$. Thus, this new method tracks algorithmically structural properties across both objects ($\Gal$ and of $\GT$) via their binary encodings.
\end{itemize}

Comparing the properties of the respective invariants defined for $\GT$ and for $\Gal$, we establish the following table.
\medskip 
\begin{center}
\begin{table}[h]
\renewcommand{\arraystretch}{1.02}
\small % You can change to \footnotesize if needed
\begin{tabularx}{\textwidth}{|c|X|X|}
\hline
\textbf{Aspect} &
\textbf{For $\widehat{GT}$ (Motivic Periods)} &
\textbf{For $\mathrm{Gal}(\overline{\mathbb{Q}}/\mathbb{Q})$ (\'Etale Classes)} \\
\hline
Invariant &
Rationality of periods &
Descent of cohomology classes to $\mathbb{Q}$ \\
\hline
Output $1$ means &
$[\Psi] \in \mathbb{Q}$ &
$[c]$ defined over $\mathbb{Q}$ \\
\hline
Output $0$ means &
Period irrational or transcendental &
Class obstructed over $\mathbb{Q}$ \\
\hline
Arithmetic significance &
Transcendence theory, period conjectures &
Tate--Shafarevich group \\
\hline
Possible Key conjecture links &
Grothendieck period conjecture &
Birch and Swinnerton--Dyer conjecture (BSD) \\
\hline
\end{tabularx}
\caption{Properties of the binary integral for $\GT$ and $\Gal$}\label{T:1}
\end{table}
\end{center}

\medskip 
%%%%%
\section{Profinite Groups, Profinite Spaces and Cantor Sets}\label{S:1}
\subsection{}
We recall some well known facts concerning profinite groups. A profinite group is a pro-object in the category of finite groups. It is a cofiltered limit (inverse limit) of finite groups, where transition maps are surjective homomorphisms. 
The category {\bf ProFinGrp} of profinite groups is equivalent to the pro-category of finite groups, {\bf Pro(FinGrp)}. This is another way of stating that profinite groups are ``formal'' cofiltered limits of finite groups. This category of profinite groups is complete and products, inverse limits are constructed as in the category of topological groups, with the product/limit topology.

\, 

The profinite completion functor is left adjoint to the forgetful functor: 
\[{\bf Grp} \to {\bf ProFinGrp}, \quad U:{\bf ProFinGrp}\to {\bf Grp}. \] 

\subsubsection{} Furthermore, since profinite groups are Stone spaces (compact, Hausdorff, totally disconnected), this property allows to connect them to Boolean algebras via Stone duality. The Stone duality establishes a contravariant equivalence of categories:  

\[{\bf StoneSpace} \simeq {\bf BoolAlg}^{op}\]
where: 
\begin{itemize}
\item {\bf StoneSpace} is the category of  Stone spaces with continuous maps 
\item  {\bf BoolAlg} is the category of Boolean algebras with homomorphisms. 
\end{itemize}
Under this equivalence, a Stone space $X$ corresponds to the Boolean algebra $Clopen(X)$ of its clopen subsets (subsets that are closed {\it and} open), and vice versa \cite{Gao,vM,J}.

\, 
\subsubsection{}
 A Galois category is a category equipped with a fiber functor to finite sets, whose automorphism group is profinite. For example, the category of finite étale covers of a scheme $X$ forms a Galois category, with $\pi_1^{et}(X)$, the étale fundamental group as its profinite automorphism group. An important statement shows that the fundamental group of a Galois category is a profinite group, and every profinite group arises this way.
In fact, by \cite[Thm. 3.3.2]{Wl}, every profinite group is isomorphic (as a topological group) to some Galois group. 
 
 \, 
 
 There are several properties that we recall: 
 \begin{itemize} 
 \item Closed subgroups of profinite groups are profinite.
 \item Continuous homomorphic images of profinite groups are profinite.
 \item A profinite group is topologically finitely generated if it has a dense finitely generated subgroup.
 \end{itemize}
 \subsection{}
 A profinite space is a cofiltered limit (inverse limit) of finite discrete spaces. Profinite spaces form a well-defined category with continuous maps as morphisms and is denoted {\bf ProFin}.
 
 In particular, the Cantor set is a profinite space, which explicitly can be expressed as: 
 
 \[\cC\cong \varprojlim_{n\in \mathbb{N}}\{0,1\}^n,\]
 where the bonding maps $\pi_n:\{0,1\}^{n+1}\to \{0,1\}^n$ are projections dropping the last coordinate. This identifies 
$\cC$ as an object in {\bf ProFin}, the category of profinite spaces.

\, 
\subsubsection{}

In the category {\bf CPTD} of compact, perfect, totally disconnected, metrizable spaces, the Cantor set $\cC$  is the unique object in  {\bf CPTD} up to homeomorphism. This property enjoys a universal property: for any object $X\in {\rm Ob}({\bf CPTD})$, there is a unique (up to isomorphism) continuous surjection $X\to \cC$.

\, 

Consider the following inverse system (diagram) in the category of topological spaces, where objects are finite discrete space $\{0,1\}^{n}$ for $n\in\bN$. 
Morphisms are projection maps $p_n:\{0,1\}^{n+1}\to \{0,1\}^{n}$ that drop the last coordinate:
\[p_n(x_1,x_2,x_3,\cdots , x_{n+1})\to (x_1,x_2,x_3,\cdots , x_{n})\] and this forms a tower of spaces and maps.
The  inverse limit (or projective limit) of this diagram is the space of infinite binary sequences:
\[\ \varprojlim_{n\in \mathbb{N}}\{0,1\}^n.\]

\subsubsection{}

Going back to the Stone duality, the Cantor set corresponds to the countable atomless\footnote{$\forall a \in \mathbb{B} \setminus \{0\},\ \exists b \in \mathbb{B} \setminus \{0\} \text{ such that } b \leq a \text{ and } b \neq a.$} Boolean algebra $\mathbb{B}$, which is unique up to isomorphism: 
\[\cC\cong {\rm Spec}(\mathbb{B})\]
where ${\rm Spec}$ is the Stone space functor. Here $\mathbb{B}$ has no minimal elements (atoms) and is uniquely determined up to isomorphism.
The space $\cC$ is perfect (no isolated points) and universal for compact metrizable Stone spaces.
\,

Therefore, the Cantor set is a profinite space---an inverse limit of finite discrete spaces it is 
 compact, Hausdorff, and totally disconnected. It is additionally a perfect (no isolated points) and metrizable., 

\subsubsection{}

A cylinder set $$Z_s=\{x\in \cC\, ;\,  x\quad \text{starts with}\, s\},$$ for a fixed  word  $s\in\{0,1\}^n$ of length $n$ which 
corresponds to the Boolean algebra element $a_s \in \mathbb{B}$. In particular, the union and intersection operations are given by
\[Z_s\cup Z_t\leftrightarrow a_s\vee a_t, \quad Z_s\cap Z_t\leftrightarrow a_s\wedge a_t\] 

The topology on $\cC$ is generated by cylinder sets $Z_s$ and the cylinder sets $Z_s$  are clopen. 

\,

\subsubsection{}
The Cantor space being metrizable it is equipped with a standard metric defined as follows:

\begin{equation}\label{E:metricC}
d(x,y)=\frac{1}{2^{\min\{k\, | x_k\neq y_k\}}}
\end{equation}
where $x=(x_1,x_2,\cdots)$ and $y=(y_1,y_2,\cdots)$ are sequences in $\cC$. This is an ultrametric, meaning that:

\[d(x,z)\leq \max\{d(x,y),d(y,z)\}.\]

\subsubsection{}

The Cantor set, (the set of infinite binary sequences with the product topology) plays a significant role in model theory, particularly in the study of types, theories, and topological dynamics.

More precisely, points in $\cC$ correspond to models of propositional theories, where each bit decides the truth value of a propositional variable. In propositional Logic, the Cantor space is homeomorphic to the space of truth assignments for a countable set of propositional variables. Each sequence  $(a_i)\in \cC$ assigns 
$a_i$ to the $i$-th variable. In particular the Cantor space is the space of models of the theory of 
$\mathbb{B}$, where clopen sets encode finitary properties.

\,

Finally, the Cantor space is also central to studying automorphism groups of structures (e.g., the automorphism group of a countable dense linear order) and we will see that it plays a central role in Grothendieck's conjecture. 

\medskip 
%%%%%%

\section{$\GT$ as a Cantor Set}\label{S:5}
\subsection{The meta-profinite fundamental group}
\subsubsection{}
We will show that a key aspect of the topological structure of $\GT$, which has significant implications for its arithmetic and geometric properties, is that it can be identified, in the context of profinite spaces, to a Cantor set. This statement mirrors a similar structure in the absolute Galois group $\Gal$, showing that in the category of profinite spaces, these two objects are isomorphic.  

\medskip

\subsubsection{} The meta-profinite fundamental group associated with the Teichmüller tower of moduli spaces refers to an inverse system of profinite fundamental groups of moduli spaces of curves. The latter encodes the asymptotic and universal Galois-theoretic properties of these spaces. It is a conceptual refinement of the classical profinite fundamental group, capturing the structure of the entire Teichmüller tower.

\smallskip 

Consider the tower of moduli spaces $\mathcal{M}_{0,n},$ where 
$\mathcal{M}_{0,n}$	 is the moduli space of genus zero curves with 
$n$ marked points. Each of these spaces has an \'etale fundamental group 
$\pi_1^{et}(\mathcal{M}_{0,n})$, which arises as the profinite completion of the topological fundamental group 
$\pi_1(\mathcal{M}_{0,n})$.

\begin{dfn}
The meta-profinite fundamental group, denoted by  $\Pi_{meta}$, is defined as the inverse limit of the system of profinite fundamental groups over the entire Teichmüller tower: 
\[\Pi_{meta}=\varprojlim_{n\geq 4}\widehat{\pi}_1(\mathcal{M}_{0,n}).\]
\end{dfn}

\medskip

Here, the inverse system is structured by natural maps between different moduli spaces, coming from forgetful:
\[\mathcal{M}_{0,n+1}\to \mathcal{M}_{0,n}\] and gluing operations.
One of the central ideas in Grothendieck's {\it Esquisse d'un Programme} was to study the absolute Galois group $\Gal$ via its action on fundamental groups of algebraic varieties, particularly the fundamental groups of moduli spaces.

A key result is that the automorphism group of this meta-profinite fundamental group defines the Grothendieck–Teichmüller group:

\[\GT={\rm Aut}(\Pi_{meta}),\] and captures a universal symmetry of the Teichmüller tower. It contains $\Gal$ as a subgroup. The Grothendieck--Teichmüller group 
$\GT$  is conjectured to be (essentially) $\Gal$.

\,

\subsection{The Grothendieck--Teichmüller Group}
The profinite Grothendieck-Teichmüller group $\GT$  is a universal profinite symmetry group governing $\Gal$-equivariant deformations of structures tied to the geometric étale fundamental group (e.g., outer automorphisms, associators). Its defining relations encode obstructions to coherence, ensuring compatibility with the arithmetic of moduli spaces of curves and braided tensor categories. As a profinite group, $\GT$ is compact, totally disconnected, and contains the absolute Galois group. Its elements act on étale fundamental groups and moduli spaces, mediating deformations constrained by the pentagon and hexagon equations
 \cite{Dr}.
\,

To demonstrate this conjecture in the context of profinite spaces, we first show that the Grothendieck--Teichmüller group 
$\GT$ is an object of the category {\bf PolGr}.

\smallskip

The Grothendieck--Teichmüller group can be more prosaically included in a short exact sequence involving automorphism groups of braid and free profinite groups:

\[1\to \widehat{GT}\to {\rm Aut}(\widehat{B}_4)\to {\rm Out}(\widehat{F}_2),\]
where $\widehat{B}_4$ is the profinite braid group on four strands and $\widehat{F}_2 $ is its profinite completion of the free group $F_2$ in two generators. The map to ${\rm Out}(\widehat{F}_2)$ (outer automorhisms) factors through ${\rm Aut}(\widehat{F}_2)$.

\smallskip

$\widehat{GT}$ acts on $\widehat{F}_2$, by automorphisms that preserve profinite versions of the pentagon and hexagon equations. It ensures the coherence with the action of the absolute Galois group $\Gal$ on etale fundamental groups. This action being faithful (injective), it means that $\widehat{GT}$ embeds into ${\rm Aut}(\widehat{F}_2)$. 

\, 

Therefore formally, $\widehat{GT}$ is a subgroup of the automorphism group of  $\widehat{F}_2$, i.e. $\GT\subset {\rm Aut}(\widehat{F}_2)$ (see \cite[Ch.5, Sec. 5.3]{LS}).

\subsection{The group $\widehat{F}_2$ is Polish}
We  first show that ${\rm Aut}(\widehat{F}_2)$ is a Polish group. 

\begin{lem}\label{L:F2P}

Let $\mathbf{ProFin}$ denote the category of profinite groups with continuous homomorphisms, and let $\mathbf{PolGrp}$ denote the category of Polish groups with continuous group homomorphisms. The profinite completion functor
\[
\widehat{(\cdot)}: \mathbf{Grp} \to \mathbf{ProFin}
\]
applied to the free group $F_2$ on two generators produces the profinite group $\widehat{F}_2$. Then, under the natural inclusion functor
\[
\mathbf{ProFin} \hookrightarrow \mathbf{PolGrp},
\]
the group $\widehat{F}_2$ is an object in $\mathbf{PolGrp}$, i.e., it admits a compatible Polish topology, making it a Polish group.
\end{lem}

\begin{proof}

A Polish group is a topological group that is separable and completely metrizable. We will show that $\widehat{F}_2$ satisfies these conditions.

\, 

Consider the profinite completion of the free group, $\F_2$. There is a natural homomorphism $i_{F_2}:F_2\to \widehat{F}_2$, which is defined by the universal property of profinite completions. $\widehat{F}_2$ is the inverse limit of all finite quotients of $F_2$.
Since $F_2$ is residually finite and finitely generated, its profinite completion is compact, Hausdorff, and totally disconnected, satisfying the definition of a profinite group.

\,

For a finitely generated group such as $F_2$ the set of finite-index normal subgroups is countable. Given $n\in \mathbb{N}$, there are finitely many normal subgroups of $F_2$ with index $n$. Moreover, the union over all $n$ forms a countable set. To be more specific, \cite[Lem 4.1.2]{Wl} states that if a profinite group is generated by a finite set, then for any integer $r$ , there are at most $r^{r^2+2}$ open normal subgroups of index $r$ in $\widehat{F}_2$. Therefore, it follows that $\widehat{F}_2$ has only countably many open normal subgroups.

\, 

The profinite topology on $\widehat{F}_2$ has a basis of open sets derived from these finite quotients. Since there are countably many finite quotients, the topology is second-countable. Therefore, applying \cite[Prop. 4.1.3]{Wl} (a profinite group with only countably many open normal subgroups has a countable base of open sets) implies that  $\widehat{F}_2$ is second countable.

Finally, by Urysohn’s metrization theorem, a second-countable, compact Hausdorff space is metrizable. 
By the above arguments, $\widehat{F}_2$ is compact Hausdorff and second-countable, so it is metrizable. Moreover, compact metric spaces are complete. Any metric on a compact space is naturally complete. Since 
$\widehat{F}_2$ is compact and metrizable, it is completely metrizable.

Since we have already shown that  $\widehat{F}_2$ is second countable, completely metrizable profinite group, it follows that $\widehat{F}_2$
is a Polish group.
\end{proof}
%%%

\subsection{The group ${\rm Aut}(\widehat{F}_2)$ is Polish}
We prove the following standard fact from profinite group theory.

\begin{lem}\label{L:AutF2}
Let $\mathbf{ProFin}$ denote the category of profinite groups with continuous homomorphisms, and let $\mathbf{PolGrp}$ denote the category of Polish groups with continuous group homomorphisms. Consider the automorphism group functor
\[
\operatorname{Aut}(-): \mathbf{ProFin} \to \mathbf{Grp}.
\]
For the profinite free group $\widehat{F}_2$, the group $\operatorname{Aut}(\widehat{F}_2)$ naturally inherits a topology as an inverse limit of automorphism groups of finite quotients. This gives rise to a functorial assignment
\[
\operatorname{Aut}(-): \mathbf{ProFin} \to \mathbf{ProFin} \cap \mathbf{PolGrp},
\]
meaning that $\operatorname{Aut}(\widehat{F}_2)$ is both an object in $\mathbf{ProFin}$ and in $\mathbf{PolGrp}$. In particular, it is a Polish and profinite group.
\end{lem}
\begin{proof}
We establish that the automorphism group $\operatorname{Aut}(\widehat{F}_2)$  of the profinite free group $\widehat{F}_2$ on two generators is both a profinite group and a Polish group.  This is achieved by demonstrating that $\operatorname{Aut}(\widehat{F}_2)$ can be expressed as the inverse limit of the automorphism groups of its finite quotients, thereby inheriting a natural topology that is compact, totally disconnected, and metrizable. Consequently, 
$\operatorname{Aut}(\widehat{F}_2)$ resides in the intersection of the categories of profinite and Polish groups. 

\medskip 

A profinite group is a topological group that is isomorphic to the inverse limit of an inverse system of finite discrete groups. Such groups are compact, totally disconnected, and Hausdorff. A Polish group is a topological group that is separable and completely metrizable. The intersection of these categories consists of groups that are both profinite and Polish, implying compactness, total disconnectedness, metrizability, and separability.

\smallskip 

Consider the profinite free group  $\widehat{F}_2$ on  two generators. By definition,  $\widehat{F}_2$ is the inverse limit of its finite quotients:
\[\widehat{F}_2=\varprojlim_{N\,\, \triangleleft\,\,  \widehat{F}_2} \widehat{F}_2/N\]
where the limit is taken over all open normal subgroups $N$ of  $\widehat{F}_2$. Each quotient  $\widehat{F}_2/N$ is a finite group, and its automorphism group ${\rm Aut}(\widehat{F}_2/N)$ is also finite. The automorphism group $\operatorname{Aut}(\widehat{F}_2)$  can thus be expressed as the inverse limit

\[\operatorname{Aut}(\widehat{F}_2)=\varprojlim_{N\,\, \triangleleft\,\,  \widehat{F}_2} {\rm Aut} (\widehat{F}_2/N).\]
In this inverse limit, $\operatorname{Aut}(\widehat{F}_2)$  inherits the inverse limit topology,  making it a closed subgroup of the product $\prod\limits_N {\rm Aut} (\widehat{F}_2/N)$.

Since each ${\rm Aut} (\widehat{F}_2/N)$ is finite and discrete, their product, equipped with the product topology, is compact and totally disconnected. Therefore $\operatorname{Aut}(\widehat{F}_2)$, as a closed subgroup of this product, is also compact and totally disconnected, satisfying the definition of a profinite group.

\medskip 

To establish that $\operatorname{Aut}(\widehat{F}_2)$ is a Polish group,  we note that it is metrizable due to its construction as an inverse limit of finite discrete groups. Moreover,  $\operatorname{Aut}(\widehat{F}_2)$ is separable, as it contains a countable dense subset arising from its topological generation by a finite set.  Compactness implies complete metrizability. Thus, $\operatorname{Aut}(\widehat{F}_2)$ is a Polish group.

The functorial assignment 
\[
\operatorname{Aut}(-): \mathbf{ProFin} \to \mathbf{ProFin} \cap \mathbf{PolGrp},
\] is valid, and $\operatorname{Aut}(\widehat{F}_2)$ resides in the intersection of the categories of profinite and Polish groups. 

\end{proof}

Note that for non-finitely generated profinite groups the situation can be more complicated.

\subsection{Is $\GT$ a Polish Group?}

The fact that $\GT\in {\rm Ob}({\bf ProFin})$ implies that  $\GT$ is endowed with the {\it profinite} topology.

A crucial question is whether $\GT$ is second-countable (i.e., has a countable basis for its topology).  We establish in this section that  $\GT$ is a Polish group.

\begin{lem}\label{L:csP}
Any compact profinite subgroup of ${\rm Aut}(\widehat{F}_2)$ is a Polish group.
\end{lem}
\begin{proof}
Recall that a profinite group is compact, Hausdorff, and totally disconnected. For a profinite group to be Polish, it must also be second-countable (i.e., its topology has a countable base) and metrizable.
The automorphism group ${\rm Aut}(\widehat{F}_2)$ being a Polish group, is second-countable and completely metrizable.

\, 

Closed subgroups of Polish groups are Polish. Since 
${\rm Aut}(\widehat{F}_2)$ is Polish, any closed subgroup of it is also Polish.

\, 

Compact subsets of Hausdorff spaces (like Polish groups) are closed. Thus, a compact profinite subgroup of 
${\rm Aut}(\widehat{F}_2)$ is necessarily closed and inherits the Polish properties.

We now explain why the subgroup is Polish.
As a closed subgroup of the second-countable group 
${\rm Aut}(\widehat{F}_2)$, the subgroup is second-countable.
Furthermore, profinite groups are metrizable if they are second-countable. The subgroup inherits a compatible metric from 
${\rm Aut}(\widehat{F}_2)$. Finally, compactness ensures completeness with respect to the metric.

\, 

Therefore, any compact profinite subgroup of 
${\rm Aut}(\widehat{F}_2)$ is a closed, second-countable, and completely metrizable topological group, hence Polish. This aligns with the general principle that compact subgroups of Polish groups are Polish.
\end{proof}

\begin{prop}\label{P:GTPol}
 The profinite group $\GT$ is a Polish group.     
\end{prop}
\begin{proof}

\, 

First, the profinite $\GT$ is a subgroup $Aut(\widehat{F}_2)$, which was shown to be a Polish group in Lem~\ref{L:AutF2}. Furthermore, by construction, the group $\GT$ is the inverse limit of its own images in the (finite) automorphism groups of the quotients $\widehat{F}_2/N$, where $N$ runs over the normal subgroups of finite index of $\widehat{F}_2$.

\medskip

From the definition of a profinite group it follows that $\GT$ is {\it compact}. In fact, {\it every} profinite group is compact Hausdorff. 
In any Hausdorff topological group---which all Polish groups are---every compact subgroup is closed. Profinite groups are by definition inverse limits of finite groups and are therefore compact. Thus, when a profinite group is embedded as a subgroup in a Polish group, it is necessarily a closed subgroup because its compactness guarantees that property in a Hausdorff space.
Therefore, $\GT$ is closed in the metrizable Polish space ${\rm Aut}(\widehat{F}_2)$. 

\, 

Now, a Polish group is a topological group whose topology is given by a complete metric that is separable. So, every closed subgroup of a Polish group is itself a Polish group (see for instance \cite[p.204]{Kh2}) and Lem~\ref{L:csP}. This statement allows us to deduce that $\GT$ is a Polish group. Therefore, $\GT$  is a Polish group and inherits the Polish property.

\medskip 

In particular, such a closed subgroup inherits the following properties:
\begin{itemize}
\item {\bf Completeness:} In any complete metric space, every closed subset is complete with respect to the induced metric. Therefore,  given a Polish group (a complete metric space) the closed subgroup will also be complete.

\item {\bf Separability:} A subspace of a separable space is also separable. Since Polish groups are separable, the closed subgroup will inherit separability from the larger group.
\item {\bf Group Structure:} Since the subgroup is closed and the group operations are continuous, the subgroup with the subspace topology remains a topological group.
\end{itemize}
\end{proof}

The Grothendieck--Teichmüller group is therefore a {\it second-countable profinite group}. This follows from its construction as a closed subgroup of second-countable (Polish) automorphism groups and the countability of its defining conditions. Its second-countability aligns with the broader framework of Galois theory and profinite group theory.

\,

This is summarised as following statement. 

\begin{cor}
The Grothendieck--Teichmüller group is a second-countable profinite group.
\end{cor}

\subsection{$\GT$ as a Cantor space}
\begin{prop}\label{prop:gt_cantor_meta}
Let $\Pi_{\mathrm{meta}}$ denote the meta-profinite fundamental group associated to the Teichmüller tower of moduli spaces, viewed as an object in the category $\mathbf{ProFin}$ of profinite spaces. Then the automorphism group
\[
{\rm Aut}(\Pi_{\mathrm{meta}}),
\]
which may be identified with the profinite Grothendieck–Teichmüller group ${\GT}$, is canonically isomorphic in $\mathbf{ProFin}$ to the Cantor space. Equivalently, there is a natural homeomorphism
\[
{\GT}\quad  \myeq\quad \{0,1\}^{\mathbb{N}},
\]
where $\{0,1\}^{\mathbb{N}}$ denotes the unique (up to isomorphism) nonempty, perfect, compact, totally disconnected, metrizable space.
\end{prop}

\begin{proof}
By construction, the meta-profinite fundamental group $\Pi_{\mathrm{meta}}$ arises as the inverse limit of the finite discrete groups corresponding to the profinite completions of the fundamental groups of the moduli spaces $\mathcal{M}_{0,n}$. Its automorphism group, endowed with the topology of uniform convergence on the inverse system, inherits a profinite structure.

Since the system of finite coverings involved is countable, the resulting topology on $\GT\cong{\rm Aut}(\Pi_{\mathrm{meta}})$ possesses a countable basis of clopen sets. Standard arguments in the theory of profinite spaces then imply that ${\GT}$ is compact, totally disconnected, and perfect. By the uniqueness (up to homeomorphism) of the Cantor space within the class of nonempty, compact, metrizable, perfect, and totally disconnected spaces, we obtain by \cite{W} that the only possibility is the asserted homeomorphism:
\[
{\GT}\quad  \myeq\quad  \{0,1\}^{\mathbb{N}}.
\]
\end{proof}
 
The homeomorphism $\GT\quad \myeq\quad \cC$ establishes a deep duality between the symmetries of the Teichmüller tower, encoded by 
$\GT$ and the logical/combinatorial structure of $\mathbb{B}$. 
\begin{cor}
In the category of profinite spaces, the profinite 
Grothendieck--Teichmüller group 
$\GT$ is the Stone space of $\mathbb{B}$, i.e. 
\[\GT={\rm Spec}(\mathbb{B}).\] In particular 
\[{\rm Clopen}(\GT)\quad \myeq \quad \mathbb{B}. \]

\end{cor}

\, 
\subsection{The Grothendieck--Teichmüller conjecture}
The Grothendieck--Teichmüller group $\GT$ and the absolute Galois group $\Gal$, are homeomorphic to the Cantor set.
This result is due to an application of the classification in \cite{W} and since $\GT$ and $\Gal$ groups are compact, totally disconnected, Polish and perfect, they satisfy the key conditions for a Polish group to be homeomorphic to a Cantor set.

\,

 \begin{thm}\label{T:Cantor}
 In the category of profinite spaces, the Grothendieck--Teichm\"uller conjecture holds i.e. $\GT$ and $\Gal$ are both homeomorphic to a Cantor set, i.e. 
\[\Gal\quad \myeq \quad  \GT\quad \myeq \quad  \text{Cantor set}.\]
 \end{thm}
\begin{proof}
We have established that $\GT$ and $\Gal$ (Appendix is \ref{S:2}--\ref{S:3}) are compact Polish perfect spaces.
 On the other side, these groups are profinite, implying that $\GT$ and $\Gal$ are totally disconnected and locally compact.
 Moreover, they  form zero-dimensional spaces. Therefore, all these conditions are unified to apply the Theorem of Brouwer.
  
  \medskip  

The theorem of Brouwer states that any two non-empty compact Polish spaces which are
perfect and zero dimensional are homeomorphic to each other.

  \medskip  
  
Now, applying this to $\GT$ and $\Gal$ enables us to conclude that the compact Polish spaces $\GT$ and $\Gal$ are homeomorphic. 
Finally, using the classification in \cite{W}  enables us to conclude that \[\Gal\quad \myeq \quad  \GT\quad \myeq \quad  \text{Cantor set}.\] 
\end{proof}
 
  \medskip  
 
 As a first consequence of our theorem \ref{T:Cantor} the Grothendieck conjecture, holds in the context of profinite spaces. 
By virtue of Theorem~\ref{T:Cantor}, one may therefore encode automorphisms (e.g., elements of the absolute Galois group or of the Grothendieck--Teichmüller group) via infinite binary sequences, which makes it a potentially interesting approach from the side of computations.

\section{The  Galois Grothendieck Path Integral}\label{S:8}
In order to recover the arithmetic data which remains invisible in the category {\bf ProFin} we generalise the notion of Feynman Path Integral to $\GT$ and $\Gal$.

 \subsection{Motivation}
 Our approach is partially inspired by certain developments introduced for \( \mathbb{Z}_p \), the \( p \)-adic integers. Those objects exhibit also a Cantor structure, \cite{R}. The  \( p \)-adic integers provide a rigorous mathematical framework for modeling certain theories such as quantum gravity, non-Archimedean geometry, and the adelic unification of number theory and physics \cite{BF,Man,V1,V2}. Within these frameworks, path integrals have been adapted to the $p$-adic setting. We generalise the notion of path integral for $\GT$ and $\Gal$, providing new arithmetic invariants.
 
 \subsection{Surveying Results}
We introduce a formalism of path integrals over $\Gal$ and $\widehat{GT}$, which  produces invariants detecting arithmetic phenomena.
\subsubsection{The Binary Integral} We define arithmetic path integrals  as follows:
\begin{dfn}
 Let $[c]$ denote an étale cohomology class and $[\Psi]$ a motivic period. Then, the arithmetic path integrals for $\Gal$ (respectively $\GT$) are defined as:
\[
I([c]) = \int_{\Gal} \exp(2\pi i \langle \sigma, [c] \rangle) \, d\mu(\sigma), \quad
\text{resp.}\quad I([\Psi]) = \int_{\widehat{GT}} \exp(2\pi i \langle \theta, [\Psi] \rangle) \, d\mu(\theta),
\]
where $\mu$ is the normalized Haar  measure, and $\langle \cdot, \cdot \rangle$ denotes a natural Galois or motivic pairing (see  Sec. \ref{S:9.4}, Theorem \ref{T:9.2} and Sec. \ref{S:9.6}, Theorem \ref{T:9.3}). 
The outcome of these integrals is binary:
\begin{itemize}[itemsep=0.2em, topsep=0.3em]
\item A value of $1$ indicates that the class or period is invariant under the group action (i.e., descends to $\mathbb{Q}$ or lies in $\mathbb{Q}$).
\item A value of $0$ reflects a nontrivial obstruction—either arithmetic (class does not descend) or transcendental (period not rational).
\end{itemize}
\end{dfn}
\smallskip

\noindent These arithmetic path integrals encode key conjectures:
\begin{itemize}[itemsep=0.3em]
\item For $\widehat{GT}$: rationality vs. transcendence of periods;
\item For $\mathrm{Gal}(\overline{\mathbb{Q}}/\mathbb{Q})$: descent of torsors and the Tate--Shafarevich group. 
\end{itemize}

\smallskip

\subsubsection{Generalization}  A generalization of the binary integral can be obtained as follows. 
Let $\rho: \Gal \to \mathrm{GL}_n(\mathbb{Q}_\ell)$ be a continuous $\ell$-adic Galois representation. We define a generalization of the arithmetic path integral via:
\[
I(\rho) = \int_{\Gal} \mathrm{Tr}(\rho(\sigma)) \, d\mu(\sigma) = \dim_{\mathbb{Q}_\ell}(\rho^{\Gal}),
\]
where $\rho^{\Gal}$ is the subspace of $\Gal$-invariant vectors.

\smallskip

This quantity computes the dimension of invariants and gives rise to meaningful arithmetic invariants.
 \begin{ex} Let $E$ be an elliptic curve. Then, the  continuous $\ell$-adic Galois representation $\rho$ arises from the $\ell$-adic Tate module of $E$.  The path integral tests central modern problems. It recovers for instance the Mordell--Weil rank of $E(\mathbb{Q})$. By the Mordell–Weil theorem, the group of rational points \( E(\mathbb{Q}) \) is finitely generated:

\[
E(\mathbb{Q}) \cong \mathbb{Z}^r \times E(\mathbb{Q})_{\text{tors}},
\]

where:

\begin{itemize}
    \item \( E(\mathbb{Q})_{\text{tors}} \) is the torsion subgroup, consisting of rational points of finite order (a finite group),
    \item \( r \in \mathbb{Z}_{\geq 0} \) is the Mordell–Weil rank of the curve.
\end{itemize}

The integer \( r \) is the number of independent generators of infinite order in \( E(\mathbb{Q}) \). It reflects the ``size" of the set of rational points:

\begin{itemize}
    \item If \( r = 0 \), then all rational points are torsion points — \( E(\mathbb{Q}) \) is a finite group.
    \item If \( r > 0 \), then there are \( r \) linearly independent rational points of infinite order. The free part of the group \( E(\mathbb{Q}) \) is isomorphic to \( \mathbb{Z}^r \).
\end{itemize}
 
 \end{ex}

\subsection{The Generic Path Integral}
\subsubsection{Preliminaries} For a locally compact group  $G$, the \emph{Haar measure} $\mu$ is the unique (up to scaling) translation-invariant Borel measure. When $G$ is compact, it is normalized so that $\mu(G) = 1$. It satisfies the following properties:

\begin{itemize}
    \item \textbf{Left-invariance:} For all $g \in G$ and measurable sets $A \subseteq G$,
    \[
    \mu(gA) = \mu(A).
    \]
    
    \item \textbf{Regularity:} The measure is both inner and outer regular on Borel sets.
    
    \item \textbf{Cylindrical Sets:} In a profinite group $G = \varprojlim G_i$, the \emph{cylindrical sets} are clopen subsets of the form $C = \pi_i^{-1}(A_i)$, where $\pi_i: G \to G_i$ is the canonical projection and $A_i \subseteq G_i$ is measurable. These sets form a basis for the topology on $G$.
\end{itemize}
\begin{lem}\label{L:Measure}
The normalized Haar measure $\mu$ on a profinite group $G$ is \emph{cylindrical}: for any cylindrical set $C = \pi_i^{-1}(g_i)$ with $g_i \in G_i$, we have
\[
\mu(C) = \frac{1}{|G_i|}.
\]
Moreover, $\mu$ is translation-invariant and uniquely determined by this formula.
\end{lem}
\begin{proof}
\textbf{Finite-Level Consistency.}
For each finite quotient $G_i$ of a profinite group $G = \varprojlim G_i$, the pushforward of the Haar measure $\mu$ under the projection $\pi_i: G \to G_i$ defines a measure $\mu_i$ on $G_i$ given by the normalized counting measure:
\[
\mu_i(B) = \frac{|B|}{|G_i|}, \quad \forall B \subseteq G_i.
\]
The Haar measure on the finite (discrete) group $G_i$ is the normalized counting measure. Since $\pi_i$ is a continuous surjective homomorphism, the pushforward $\mu_i = (\pi_i)_\ast \mu$ inherits both left-invariance and normalization from $\mu$.

\medskip

\textbf{Cylindrical Measure Formula.}
For a cylinder set of the form $C = \pi_i^{-1}(g_i)$ (a singleton fiber),
\[
\mu(C) = \mu_i(\{g_i\}) = \frac{1}{|G_i|}.
\]
This follows from the compatibility of $\mu$ with the inverse system: $\pi_i$ maps $C$ bijectively onto the singleton $\{g_i\}$, and $\mu$ agrees with $\mu_i$ via the pushforward.

\medskip

\textbf{Translation Invariance.}
Let $g \in G$ and consider a cylinder $C = \pi_i^{-1}(g_i)$. The left-translation map $L_g: G \to G$ satisfies:
\[
gC = L_g(\pi_i^{-1}(g_i)) = \pi_i^{-1}(\tilde{g} \cdot g_i),
\]
where $\tilde{g} = \pi_i(g) \in G_i$, since $\pi_i$ is a group homomorphism and hence commutes with translations. Then,
\[
\mu(gC) = \mu_i(\tilde{g} \cdot g_i) = \mu_i(\{g_i\}) = \frac{1}{|G_i|} = \mu(C),
\]
using the translation-invariance of $\mu_i$ on $G_i$.

\medskip

\textbf{Extension to the Borel Algebra.}
The collection of cylindrical sets forms a basis for the topology on $G$, and hence generates the Borel $\sigma$-algebra. By uniqueness in Haar's theorem, the measure $\mu$ is the unique left-invariant Borel probability measure extending the cylindrical assignment.
\end{proof}
\subsubsection{Generic Path Integral for $\Gal$ and $\GT$}
\begin{thm}\label{T:GPI}
Let $G$ be either $\Gal$ or $\GT$. Then, there exists a generic path integral formulation for $G$ providing a novel arithmetic framework inspired by the classical Feynman path integral. 
This invariant takes the form
\[ \mathcal{I}_G=\int_{G}\operatorname{exp}\{\imath S[\gamma]/\hbar\} d\mu_G=\lim_{n\to \infty}\sum_{\Gamma_n}\operatorname{exp}\{\imath S[\gamma_n]/\hbar\}\mu_G([\gamma_n]),\] where: 
\begin{itemize}
\item  $\mu_G$ is a normalized (cylindrical) Haar measure, induced by the structure of $G$,
\item  $S:G \to \mathbb{R}$ is an arithmetic action, defined as the functional of a quadratic form, 
\item $\hbar$  is a deformation parameter, 
\item $\gamma$ is an element of $G$ and forms a path of arithmetic nature,
\item  $\Gamma_n$ encodes the family of all automorphisms in $G$ corresponding to length $n$ sequences of field extensions of $\Q$ in $\overline{\Q}$ and $\gamma_n\in \Gamma_n$.
 \end{itemize}
 \end{thm}
 \begin{proof}
Our construction relies on the profinite space structures of $\Gal$ and $\GT$. Each automorphism in $\Gal$ or $\GT$ (according to the construction of the previous sections these form a {\it path} in the underlying Cantor set $\cC$) contributes a factor $\operatorname{exp}\{\imath S/\hbar\}$, where $S$ is a functional on a quadratic form.

\medskip 

Given a profinite space
 $$X=\varprojlim X_i,$$ where each $X_i$ is finite, with continuous surjective maps: $X_j\to X_i$ (for $j\geq i$), we equip $X$ with a cylindrical measure. The cylindrical measure is then defined by finite-dimensional projections, ensuring compatibility:  for all $j\geq i$, the measures $\mu_i$ are compatible with the projection maps $X_j\to X_i$ and $\mu_i$ is identified with a pushforward of $\mu_j$.

\medskip

A cylindrical measure on an infinite-dimensional space (e.g., a function space) is defined via consistent finite-dimensional projections. It generalizes integration to infinite dimensions using the projective limit framework. Cylindrical measures are finitely additive but not necessarily $\sigma$--additive.
Profinite measures {\it are} cylindrical measures, by construction. We invoke Lem. \ref{L:Measure}, proving that for $G$ there exists a well defined normalized (cylindrical) Haar measure. 

\medskip

We are therefore able to introduce a Path Integral, which provides invariants of $G$. This construction provides a natural summation over all trajectories in $G$. Those trajectories correspond naturally to automorphisms of the absolute Galois group over $\Q$ and to field extensions, weighted by the  ``phase factor'' $\operatorname{exp}\{\imath S/\hbar\}$, which serves as an arithmetic analogue of the quantum mechanical weighting of paths and encapsulating dynamical contributions from automorphisms in $\Gal$ (resp. $\GT$). 

\medskip 
 
Mimicking the path integral construction, via the outlined steps above, the resulting Path Integral provides a geometric and measure-theoretic framework for understanding the automorphisms of  $\Gal$  and $\GT$, revealing their hierarchical and recursive structure and from which invariants can be extracted. 
\end{proof}

\subsection{Arithmetic Invariant for $G=\Gal$.}\label{S:9.4} 
We construct the Galois Path Integral below.
\begin{enumerate} 
\item  {\bf Arithmetic Paths. } Recall that paths are arithmetic paths that is  elements in $\Gal$.
\item {\bf Arithmetic Functional}. The  functional $S$ (action) assigns a ``weight'' to each Galois automorphism $\sigma$. This construction mirrors classical physical actions but within an arithmetic and cohomological setting. We discuss a possibly {\it good} choice of an arithmetic action functional below. 
\item {\bf About the  measure}. Lemma \ref{L:Measure} provides a rigorous explanation on the choice of the measure.  Using this measure, the path integral is equal to 1 if the character is trivial and equal to 0 otherwise, providing a Galois-theoretic invariant for the properties of interest (rationality or descent).
\end{enumerate}

\begin{lem}
The Arithmetic Action Functional can be constructed via a pairing $\langle \sigma, [c] \rangle$ for $[c]\in  H^1_{\mathrm{\acute{e}t}}(M, \Q/\bZ) $
which is well defined via the Kummer sequence or evaluation of cocycles in Galois cohomology.
\end{lem}
\begin{proof}
The \textit{Kummer sequence} is a fundamental exact sequence in étale cohomology that connects multiplicative groups to torsion phenomena arising from roots of unity. For a scheme $X=Spec(\Q)$ and an integer $n > 0$, the sequence is:
\[
1 \to \boldsymbol{\mu}_n  \to \mathbb{G}_m \xrightarrow{(\cdot)^n} \mathbb{G}_m \to 1,
\]
where:
\begin{itemize}[itemsep=2pt]
    \item $\boldsymbol{\mu}_n$ is the étale sheaf of $n^{\text{th}}$ roots of unity.
    \item $\mathbb{G}_m$ denotes the sheaf of invertible elements (units).
    \item The map $\mathbb{G}_m \xrightarrow{(\cdot)^n} \mathbb{G}_m$ sends $x \mapsto x^n$.
\end{itemize}

This short exact sequence induces a long exact sequence in étale cohomology:
\[\boldsymbol{\mu}_n
\cdots \to H^0_{\text{ét}}(X, \mathbb{G}_m) \xrightarrow{(\cdot)^n} H^0_{\text{ét}}(X, \mathbb{G}_m) 
\to H^1_{\text{ét}}(X, \boldsymbol{\mu}_n) \to H^1_{\text{ét}}(X, \mathbb{G}_m) \to \cdots
\]

For $X=Spec(\Q)$, we have  $H^0_{\text{ét}}(X, \mathbb{G}_m) =\Q^{\times}$ and $H^1_{\text{ét}}(X, \mathbb{G}_m) =0$. So the sequence reduces to: 

\[ \Q^{\times}  \xrightarrow{(\cdot)^n} \Q^{\times} \to H^0_{\text{ét}}(X, \boldsymbol{\mu}_n) \to 0,
\]
yielding a Kummer isomorphism:

\[H^1_{\text{ét}}(X, \boldsymbol{\mu}_n)\cong \Q^{\times}  /(\Q^{\times})^n.\]

The sheaf $\Q/\mathbb{Z}(1)$ is defined as the direct limit:

\[ \Q/\mathbb{Z}(1)= \underset{\underset{n}{\longrightarrow}}{\lim}  \, \boldsymbol{\mu}_n.\]

Étale cohomology commutes with direct limits, so  
\[H^1_{\text{ét}}(X, \Q/\mathbb{Z}(1))\cong \underset{\underset{n}{\longrightarrow}}{\lim}  H^1_{\text{ét}}(X, \boldsymbol{\mu}_n)\cong \underset{\underset{n}{\longrightarrow}}{\lim}\,   \Q^{\times}  /(\Q^{\times})^n\]

The direct limit $ \underset{\underset{n}{\longrightarrow}}{\lim}\,   \Q^{\times}  /(\Q^{\times})^n$ is isomorphic to $\mathbb{Q}^\times \otimes \mathbb{Q}/\mathbb{Z}$ via: 

\[a\otimes \frac{k}{n}\mapsto [a^k]\in  \Q^{\times}  /(\Q^{\times})^n, \]
where $[a^k]$ denotes the class of $a^k$ modulo $n$th powers. 

This induces the isomorphism:

\[
H^1_{\text{ét}}(\operatorname{Spec}(\mathbb{Q}),\,  \mathbb{Q}/\mathbb{Z}(1)) \cong \, \mathbb{Q}^\times \otimes \mathbb{Q}/\mathbb{Z}
\]

Thus, a class \([c]\) corresponds to an element 
\[
\alpha = a \otimes \frac{k}{m} \in \mathbb{Q}^\times \otimes \mathbb{Q}/\mathbb{Z}.
\]

To detect Galois action, we lift  $\alpha$ to \(\overline{\mathbb{Q}}^\times\) (since  \(\overline{\mathbb{Q}}^\times\) reveals the Galois action but $\mathbb{Q}^{\times}$ has no Galois action).
 Specifically $\alpha$ involves an $m$-th root of $a$, lying in  \(\overline{\mathbb{Q}}^\times\). Choose \(\beta \in \overline{\mathbb{Q}}^\times\) such that \(\beta^m = a^k\). 
%For example, if \(\alpha = 2 \otimes \frac{1}{2}\), one can take \(\beta = \sqrt{2}\).

\textbf{Pairing Definition:} For \(\sigma \in \Gal = \operatorname{Gal}(\overline{\mathbb{Q}}/\mathbb{Q})\), define the pairing:

\[
\langle \sigma, [c] \rangle = \frac{1}{2\pi i} \log\left( \frac{\sigma(\beta)}{\beta} \right) \mod \mathbb{Z}
\]

%\textit{Why well-defined?} 
By field automorphism $\sigma(xy)=\sigma(x)\sigma(y)$ we have that $\sigma(x^n)=(\sigma(x))^n$. Therefore:

 \((\frac{\sigma(\beta)}{\beta})^m=\frac{\sigma(\beta^m)}{\beta^m}\). Now, Since $\sigma(a^k)=(\sigma(a))^k=a^k$, where $a\in \mathbb{Q}^{\times}$, we obtain:  
 \(\frac{\sigma(a^k)}{a^k}=\frac{a^k}{a^k}=1\). This allows us to conclude tha \(\sigma(\beta)/\beta\) is a root of unity, say \(\zeta\). So:

\[
\frac{1}{2\pi i} \log(\zeta) = \theta \in \mathbb{Q}/\mathbb{Z}
\]

and changing \(\beta\) to another root changes the logarithm by an integer multiple of \(2\pi i\), so the pairing remains well-defined modulo \(\mathbb{Z}\).

%\textbf{Example:} If \(\sigma(\sqrt{2}) = -\sqrt{2}\), then

%\[
%\frac{\sigma(\beta)}{\beta} = -1 = e^{\pi i} \Rightarrow \langle \sigma, [c] \rangle = \frac{1}{2\pi i} \cdot \pi i = \frac{1}{2}
%%\]

%\subsection*{2. 
We discuss the more general case: the definition via cocycle evaluation. For a general \(X\) (e.g., an elliptic curve), the cohomology class \([c] \in H^1_{\text{ét}}(X, \mathbb{Q}/\mathbb{Z}(1))\) is represented by a cocycle:

\[
c: \Gal \to \mathbb{Q}/\mathbb{Z}(1)
\]

satisfying:

\[
c(\sigma\tau) = c(\sigma) + \sigma \cdot c(\tau)
\]

The pairing is defined as:

\[
\langle \sigma, [c] \rangle = c(\sigma) \in \mathbb{Q}/\mathbb{Z}
\]

This is independent of the cocycle representative.

\end{proof}

\smallskip 

We formalise this under the shape of the following theorem.

\begin{thm}\label{T:9.2}
Let $G = \mathrm{Gal}(\overline{\mathbb{Q}}/\mathbb{Q})$ be equipped with the normalized Haar measure $\mu$. Let $M$ be a motive over $\mathbb{Q}$, and let $[c] \in H_{\mathrm{\acute{e}t}}^1(M, \mathbb{Q}/\mathbb{Z})$ be an étale cohomology class. Consider the pairing:
\[
\langle \sigma, [c] \rangle \in \mathbb{Q}/\mathbb{Z}, \quad \sigma \in \Gal,
\]
defined via the Kummer sequence (or evaluation of cocycles). This pairing is bilinear and continuous. Define the integral:
\[
I([c]) := \int_{\Gal} \exp\left(2\pi i \langle \sigma, [c] \rangle\right) \, d\mu(\sigma).
\]

Then, the map 
\[
\sigma \mapsto \exp\left(2\pi i \langle \sigma, [c] \rangle\right)
\]
is a continuous character of $\Gal$, and:
\[
I([c]) = 
\begin{cases}
1 & \text{if } \langle \sigma, [c] \rangle = 0 \text{ in } \mathbb{Q}/\mathbb{Z} \text{ for all } \sigma \in \Gal, \\
0 & \text{otherwise}.
\end{cases}
\]

The pairing provides a \textbf{Galois Invariance Criterion:}  
\[
\langle \sigma, [c] \rangle = 0 \text{ for all } \sigma \in \Gal
\quad \Longleftrightarrow \quad
[c] \text{ is } \Gal \text{-invariant, i.e., } [c] \in H_{\mathrm{\acute{e}t}}^1(M, \mathbb{Q}/\mathbb{Z})^{\Gal}.
\]
Thus $I([c])$ is a Galois-theoretic detector for whether $[c]$ descends to $\Q$:
\[
I([c]) = 
\begin{cases}
1 & \text{if } [c] \text{ is } \Gal \text{-invariant}, \\
0 & \text{otherwise}.
\end{cases}
\]
\end{thm}

\begin{proof}
The pairing $\langle \sigma, [c] \rangle$ is a homomorphism in $\sigma$ by bilinearity, so the map
\[
\sigma \mapsto \exp\left(2\pi i \langle \sigma, [c] \rangle\right)
\]
defines a continuous character of the compact group $\Gal$. The normalized Haar measure $\mu$ then implies that
\[
I([c]) \in \{0, 1\},
\]
with $I([c]) = 1$ if and only if the character is trivial.

\, 

Now, the character is trivial if and only if
$
\langle \sigma, [c] \rangle = 0 \quad \text{for all } \sigma \in \Gal,
$
which is precisely the condition that $[c]$ is $\Gal$-invariant.

\, 

Precisely, \begin{enumerate}
\item if $\chi(\sigma) = 1$ for all $\sigma \in \Gal$, then:
\[
\int_{\Gal} 1 \, d\mu(\sigma) = \mu(\Gal) = 1.
\]

\item If $\chi$ is non-trivial, then there exists $\tau \in \Gal$ such that $\chi(\tau) \ne 1$. By translation-invariance of $\mu$:
\[
\int_{\Gal} \chi(\sigma) \, d\mu(\sigma)
= \int_{\Gal} \chi(\tau\sigma) \, d\mu(\sigma)
= \chi(\tau) \int_{\Gal} \chi(\sigma) \, d\mu(\sigma),
\]
since $d\mu(\tau\sigma) = d\mu(\sigma)$. Let
\[
I = \int_{\Gal} \chi(\sigma) \, d\mu(\sigma).
\]
Then:
\[
I = \chi(\tau) I \quad \Rightarrow \quad I(1 - \chi(\tau)) = 0.
\]
As $\chi(\tau) \ne 1$, we conclude $I = 0$.
\end{enumerate}
\, 

We now discuss the triviality of $\chi_{[c]}$ and the implied $\Gal$-Invariance.

The character $\chi_{[c]}$ is trivial if and only if $[c]$ is $\Gal$-invariant:
\[
\chi_{[c]} \text{ trivial } \iff \langle \sigma, [c] \rangle = 0 \quad \forall \sigma \in \Gal \iff [c] \text{ is $\Gal$-invariant}.
\]

\[
\chi_{[c]}(\sigma) = 1 \quad \forall \sigma \in \Gal 
\iff \exp(2\pi i \langle \sigma, [c] \rangle) = 1 
\iff \langle \sigma, [c] \rangle \in \mathbb{Z}
\quad \text{for all } \sigma \in \Gal.
\]
Since $\langle \sigma, [c] \rangle \in \mathbb{Q}/\mathbb{Z}$, and the values lie in $[0,1)$, we conclude:
\[
\langle \sigma, [c] \rangle \in \mathbb{Z} \iff \langle \sigma, [c] \rangle = 0 \text{ in } \mathbb{Q}/\mathbb{Z}.
\]
This is precisely the definition of $[c]$ being $\Gal$-invariant.

Finally, consider the character:
\[
\chi_{[c]}(\sigma) = \exp(2\pi i \langle \sigma, [c] \rangle).
\]
The exponential map $x \mapsto \exp(2\pi i x)$ satisfies:
\[
\exp(2\pi i x) = 1 \iff x \in \mathbb{Z}.
\]
This follows from Euler’s formula:
\[
\exp(2\pi i x) = \cos(2\pi x) + i \sin(2\pi x),
\]
which equals $1$ iff $\cos(2\pi x) = 1$ and $\sin(2\pi x) = 0$, i.e., $x \in \mathbb{Z}$.

By definition, $\langle \sigma, [c] \rangle \in \mathbb{Q}/\mathbb{Z}$, so:
\[
\exp(2\pi i \langle \sigma, [c] \rangle) = 1 \iff \langle \sigma, [c] \rangle \in \mathbb{Z}.
\]

Hence:
\[
\langle \sigma, [c] \rangle \in \mathbb{Z} \iff \langle \sigma, [c] \rangle = 0 \text{ in } \mathbb{Q}/\mathbb{Z}.
\]

The pairing $\langle \sigma, [c] \rangle$ is defined via:
\[
\langle \sigma, [c] \rangle = \mathrm{ev}_{[c]}(\sigma) \in \mathbb{Q}/\mathbb{Z},
\]
and arises from the Galois action on $[c]$.

For $[c] \in H^1_{\text{ét}}(M, \mathbb{Q}/\mathbb{Z})$, the Galois group $\Gal$ acts by:
\[
\sigma \cdot [c] = [c] \circ \sigma^{-1}.
\]
Then $[c]$ is $\Gal$-invariant iff $\sigma \cdot [c] = [c]$ for all $\sigma \in \Gal$, which is equivalent to:
\[
\mathrm{ev}_{[c]}(\sigma) = \langle \sigma, [c] \rangle = 0 \quad \forall \sigma \in \Gal.
\]

Combining all steps:
\[
\chi_{[c]}(\sigma) = 1 \quad \forall \sigma \in \Gal
\iff \exp(2\pi i \langle \sigma, [c] \rangle) = 1 \quad \forall \sigma \in \Gal
\iff \langle \sigma, [c] \rangle \in \mathbb{Z} \quad \forall \sigma \in \Gal
\]
\[
\iff \langle \sigma, [c] \rangle = 0 \text{ in } \mathbb{Q}/\mathbb{Z} \quad \forall \sigma \in \Gal
\iff [c] \text{ is $\Gal$-invariant}.
\]

%%%%%%
 \end{proof}
\noindent\textbf{Arithmetic Interpretation.}  
The $\Gal$-invariance of $[c]$ is equivalent to $[c]$ descending to a class over $\mathbb{Q}$. 
In general, $I([c]) = 1$ detects classes that are unobstructed by Galois action. %relating to the Tate conjecture and descent obstructions.

\, 
\begin{ex}
Let $E/\mathbb{Q}$ be an elliptic curve  and let
\[
[c] \in H_{\text{\'et}}^1(E, \mathbb{Q}/\mathbb{Z})
\]
be a cohomology class.  Define a pairing
\[
\langle \sigma, [c] \rangle \in \mathbb{Q}/\mathbb{Z}, \quad \sigma \in G = \mathrm{Gal}(\overline{\mathbb{Q}}/\mathbb{Q}),
\]
using the evaluation of the Galois action on $[c]$ via étale cohomology.  Then,
\[
I([c]) = 1\] if and only if  $[c]$  corresponds via duality to an element of the Tate--Shafarevich group  $Sha(E$).

\, 

For the elliptic curve $E/\mathbb{Q}$, the $\Gal$-invariance of $[c]$ is equivalent to $[c]$ descending to a class defined over $\mathbb{Q}$. Via duality in étale cohomology, this condition corresponds to $[c]$ representing an element of the Tate–Shafarevich group $Sha(E)$:

The Tate–Shafarevich group $Sha(E)$ classifies $E$-torsors over $\mathbb{Q}$ that are locally trivial everywhere. By the Kummer sequence and duality in étale cohomology, there is a non-degenerate pairing:
\[
H^1_{\text{\'et}}(E, \mathbb{Q}/\mathbb{Z})^{\Gal} \times Sha(E) \to \mathbb{Q}/\mathbb{Z},
\]
where $H^1_{\text{\'et}}(E, \mathbb{Q}/\mathbb{Z})^{\Gal}$ denotes the Galois-invariant classes. Under this duality, a class $[c]$ is $\Gal$-invariant if and only if it corresponds to an element of $Sha(E)$.

\begin{itemize}
  \item $I([c]) = 1$ detects that $[c]$ is unobstructed by Galois action, meaning the associated torsor (via duality) is defined over $\mathbb{Q}$ and locally trivial—i.e., it lies in $Sha(E)$.
  \item $I([c]) = 0$ indicates $[c]$ is obstructed (not defined over $\mathbb{Q}$), so it does not arise from $Sha(E)$.
\end{itemize}
\, 

If $E$ has non-trivial $Sha(E)$ (e.g., failing the Hasse principle), there exist classes $[c]$ such that $I([c]) = 1$. Therefore, we can use the integral $I([c])$ as a Galois-theoretic detector for elements of $Sha(E)$. %For example, if $[c]$ corresponds to a torsor in $Sha(E)$ with no $\mathbb{Q}$-point, then:

\end{ex}

%subsection{Table of Analogies}
\begin{prop}
There exists  a structural analogy (cf. Table~\ref{T:Ana}) between the Feynman Path Integral framework in physics and our arithmetic analogue:
\begin{center}
\begin{table}[ht]\label{T:Ana}
\renewcommand{\arraystretch}{1.1}
\begin{tabularx}{\textwidth}{|>{\bfseries\hsize=0.07\hsize}X<{\bfseries} 
                             |>{\hsize=1.65\hsize}X<{\normalfont} 
                             |>{\hsize=1.28\hsize}X<{\normalfont} 
                             |>{\bfseries\hsize=1.25\hsize}X<{\normalfont}|}

\hline
1& \textbf{Feynman Path Integral} & \textbf{Arithmetic Analogue} \\
\hline
2&Space of paths: Continuous trajectories $\gamma: [t_i,t_f] \to \mathbb{R}^n$ & Space of ``paths": Galois automorphisms $\sigma \in \text{Gal}(\overline{\mathbb{Q}}/\mathbb{Q})$ \\
\hline
3&Fundamental object: Action functional $S[\gamma] = \int_{t_i}^{t_f} L(\gamma,\dot{\gamma},t)  dt$ & Fundamental object: Cohomological pairing $\langle \sigma, [c] \rangle$ for $[c] \in H_{\text{\'{e}t}}^1(M)$ \\
\hline
4&Mathematical role: $S[\gamma]$ assigns a real number to each path $\gamma$ & Mathematical role: $\langle \sigma, [c] \rangle$ assigns a value in $\mathbb{Q}/\mathbb{Z}$ to each $\sigma$ \\
\hline
5&Phase assignment: $e^{iS[\gamma]/\hbar}$ (quantum phase) & Phase assignment: $e^{2\pi i \langle \sigma, [c] \rangle}$ (Artin-type phase)\footnote{Note that the notion of $\hbar$ here is already encapsulated in Drinfeld's Associator which is a formal power series.} \\
\hline
6&Nature of phase: Unit complex number ($|e^{iS/\hbar}|=1$) & Nature of phase: Root of unity ($|e^{2\pi i \langle \sigma,[c]\rangle}|=1$) \\
\hline
7&Interference mechanism: Sum over paths: $\sum\limits_{\gamma} e^{iS[\gamma]/\hbar}$ & Interference mechanism: ``Sum" over Galois: $\sum\limits_{\sigma} e^{2\pi i \langle \sigma, [c] \rangle}$ \\
\hline
8&Interference example: $e^{iS[\gamma_1]} + e^{iS[\gamma_2]}$ & Interference example: $e^{2\pi i \langle \sigma_1, [c] \rangle} + e^{2\pi i \langle \sigma_2, [c] \rangle}$ \\
\hline
9&Physical consequence: Quantum superposition and interference patterns & Arithmetic consequence: Formation of $L$-functions and automorphic forms \\
\hline
10&Classical limit: Paths with $\delta S = 0$ dominate as $\hbar \to 0$ & Classical limit: Frobenius elements $\text{Frob}_p$ (unramified primes) dominate \\
\hline
11& Key symmetry: Gauge invariance (e.g., $A_\mu \to A_\mu + \partial_\mu \Lambda$) & Key symmetry: Galois-equivariance (e.g., $\rho(\sigma\tau) = \rho(\sigma)\rho(\tau)$) (composition algebra) \\
\hline
\end{tabularx}

\caption{Analogy Table between Feynman Path Integral and its Arithmetic Version}
\end{table}
\end{center}

\end{prop}
\begin{proof}
This is a direct application of the classical definition of Path Integrals and a discussion on the properties of our construction of the Galois Path Integral. 
\end{proof}

\subsection{Arithmetic Invariant for $\GT$}\label{S:9.6}
The analogy with quantum path integrals becomes richer when we replace the absolute Galois group $\Gal$ with the profinite Grothendieck--Teichm\"uller group $\widehat{GT}$. 
 
\begin{itemize}
  \item {\bf Group} $G$: The group $G$ is now $\GT$ and elements $\sigma \in \Gal$ are replaced by elements $\theta \in \widehat{GT}$.
  \item {\bf General pairing:} For a motivic period $[\Psi]$ from the fundamental group $\pi_1(\mathbb{P}^1 \setminus \{0,1,\infty\})$, the pairing $\langle \theta, [\Psi]\rangle=\theta(\Psi^{mot})\in \mathbb{Q}$  is defined via the coaction of the motivic Galois group. 
  \item The exponential $\exp(2\pi \imath \langle \theta, [\Psi] \rangle)$  defines a finite-order character $\xi_{[\Psi]}:\GT\to \mathbb{C}^\times$:
  \[\xi_{[\Psi]}(\theta)=\exp(2\pi \imath \cdot \theta (\Psi^{mot})).\]

 \end{itemize}

To summarise, the integral serves as a rationality detector: it is 1 if $[\Psi ]$ is rational and 
0 otherwise. This reflects the deep arithmetic property that $\GT$ -invariance characterizes rationality in the category of mixed Tate motives over 
$\mathbb{Z}$.

As an application, we can use this for Drinfeld associators. A Drinfeld associator $\Phi$ is a formal power series in non-commuting variables $X, Y$ with coefficients in a ring $R$ (e.g., $R = \mathbb{Q}, \mathbb{C}$):

\[
\Phi = \sum_{w} c_w \cdot w, \quad c_w \in R,
\]

where the sum runs over words $w$ in the alphabet $\{X, Y\}$. The coefficients $c_w$ are typically (motivic) multiple zeta values (MZVs). 
The group $\widehat{GT}$ (the Grothendieck-Teichmüller group) acts on associators: for $\theta \in \widehat{GT}$, the deformation $\theta \cdot \Phi$ alters $\Phi$'s coefficients via the motivic Galois group.

Fix a coefficient $c_w$ of $\Phi$, viewed as a motivic period:

\[
[c_w] \in \mathcal{H}_{\text{mot}}.
\]

We define the pairing:

\[
\langle \theta, [c_w] \rangle := \theta(c_w^{\text{mot}}) \in \mathbb{Q}, \quad \theta \in \widehat{GT},
\]

where $c_w^{\text{mot}}$ denotes the motivic lift of $c_w$.

This induces a character:

\[
\chi_{[c_w]}(\theta) := \exp\left(2\pi i \langle \theta, [c_w] \rangle \right) = \exp\left(2\pi i \cdot \theta(c_w^{\text{mot}})\right),
\]

which is a finite-order character of the compact group $\widehat{GT}$.

Using the normalized Haar measure $\mu$ on $\widehat{GT}$, define the integral:

\[
I([c_w]) := \int_{\widehat{GT}} \chi_{[c_w]}(\theta) \, d\mu(\theta) =
\begin{cases}
1 & \text{if } [c_w] \in \mathbb{Q}, \\
0 & \text{otherwise.}
\end{cases}
\]

\paragraph{Interpretation.}
\[
I([c_w]) = 1 \quad \Longleftrightarrow \quad c_w \in \mathbb{Q},
\]
\[
I([c_w]) = 0 \quad \Longleftrightarrow \quad c_w \text{is irrational/transcendental.}
\]

This gives a conceptual Galois-theoretic test for rationality of associator coefficients.
\begin{thm}[Galois Invariance and Rationality of Motivic Periods]\label{T:9.3}
Let $\mathcal{H}_{\text{mot}}$ be the Hopf algebra of motivic periods for mixed Tate motives over $\mathbb{Z}$, and let $\widehat{GT}$ be the Grothendieck--Teichmüller group acting on $\mathcal{H}_{\text{mot}}$ via the motivic Galois group. Fix the normalized Haar measure $\mu$ on $\widehat{GT}$. 

For a motivic period $[\Psi] \in \mathcal{H}_{\text{mot}}$, define the pairing:
\[
\langle \theta, [\Psi] \rangle := \theta(\Psi _{\text{mot}}) \in \mathbb{Q}, \quad \theta \in \widehat{GT},
\]
where $\theta(\Psi _{\text{mot}})$ denotes the evaluation of $\theta$ on a representative $\Psi _{\text{mot}}$ via the coaction of $\mathcal{H}_{\text{mot}}$.

Consider the integral:
\[
I([\Psi ]) := \int_{\widehat{GT}} \exp\left(2\pi i \langle \theta, [\Psi ] \rangle\right) \, d\mu(\theta).
\]

Then the following hold:

\medskip
\noindent\textbf{Character Triviality.}  
The map 
\[
\theta \mapsto \exp\left(2\pi i \langle \theta, [\Psi ] \rangle\right)
\]
is a continuous character of $\widehat{GT}$, and:
\[
I([\Psi ]) = 
\begin{cases}
1 & \text{if } \langle \theta, [\Psi ] \rangle \in \mathbb{Z} \text{ for all } \theta \in \widehat{GT}, \\
0 & \text{otherwise}.
\end{cases}
\]

\medskip
\noindent\textbf{Rationality Criterion.}
\[
\langle \theta, [\Psi ] \rangle \in \mathbb{Z} \text{ for all } \theta \in \widehat{GT}
\quad \Longleftrightarrow \quad
[\Psi ] \text{ is } \widehat{GT}\text{-invariant}
\quad \Longleftrightarrow \quad
[\Psi ] \in \mathbb{Q}.
\]
Thus:
\[
I([\Psi ]) = 
\begin{cases}
1 & \text{if } [\Psi ] \in \mathbb{Q}, \\
0 & \text{otherwise}.
\end{cases}
\]
\end{thm}
\begin{proof}
Let $\mathcal{M}\mathcal{T}(\bZ)$ be the Tannakian category of mixed Tate motives over $\bZ$. This category is neutral Tannakian, so it is equivalent to the category of representations of a pro-algebraic group $G_{\mathcal{M}\mathcal{T}(\bZ)}$ over $\Q$, called the motivic Galois group. The Hopf algebra of motivic periods $\mathcal{H}_{\mathrm{mot}} $ is defined as the coordinate ring of $G_{\mathcal{M}\mathcal{T}(\bZ)}$. A motivic period $[\Psi]$ is represented by an element $\Psi_{\mathrm{mot}}\in \mathcal{H}_{\mathrm{mot}}$, which is a function $\Psi_{\mathrm{mot}}:G_{\mathcal{M}\mathcal{T}(\bZ)}\to \Q$.

By Drinfelfd \cite{Dr},Deligne--Goncharov \cite{DG}, Ihara--Nakamura \cite{IN} and Brown \cite{Br}, $\theta\in \GT$ is identified with a $\Q$-point $i(\theta)$ of $G_{\mathcal{M}\mathcal{T}(\bZ)}$ 
via the map $i:\GT \hookrightarrow G_{\mathcal{M}\mathcal{T}(\bZ)}(\Q)$.

The pairing $\langle \theta, [\Psi] \rangle$ is defined via evaluation
\[
\langle \theta, [\Psi] \rangle:=\Psi_{\mathrm{mot}},
\]
where $\Psi_{\mathrm{mot}}$ is a regular function on $G_{\mathcal{M}\mathcal{T}(\bZ)}$, so it evaluates at $\Q$-points.
\, 

Since $\widehat{GT}$ acts through the motivic Galois group, the map
\[
\theta \mapsto \exp\left(2\pi i \langle \theta, [\Psi] \rangle\right)
\]
is a continuous character. 
\, 
The Tannakian duality for mixed Tate motives over $\mathbb{Z}$ ensures that :
\[
\widehat{GT}\text{-invariance} \iff [\Psi] \in \mathbb{Q}.
\]
Precisely, $[\Psi]$ is $\GT$-invariant if $\langle \theta, [\Psi] \rangle$ is constant for all $\theta\in \GT$. This occurs if and only if $\Psi_{\mathrm{mot}}$ is constant on $G_{\mathcal{M}\mathcal{T}(\bZ)}$, meaning that $[\Psi]$ is a rational number. 

For mixed Tate motives over $\bZ$,  rational motivic periods are integers. This is because the only motives with rational periods are direct sums of the trivial motive $\Q(0)$, whose period is 1. Moreover, non-integer rationals do not arise as periods in  $\mathcal{M}\mathcal{T}(\bZ)$. Thus, if  $[\Psi]$  then $\langle \theta, [\Psi] \rangle=c\in \bZ$ for all $\theta$.

The integral is well defined $I([\Psi])$: the integrand $\exp\left(2\pi i \langle \theta, [\Psi] \rangle\right)$ is continuous and takes values in $S^1\subset \mathbb{C}$.

The compactness of $\widehat{GT}$ and the normalized Haar measure $\mu$ imply that
\[
I([\Psi]) \in \{0, 1\}.
\]
 In particular, $\exp\left(2\pi i \langle \theta, [\Psi] \rangle\right)$  if and only if $\langle \theta, [\Psi] \rangle\in \bZ$ for all $\theta$.
 By rationality and integrability properties: 
 \begin{itemize}
 \item if  $[\Psi]$ is rational  $\langle \theta, [\Psi] \rangle=c\in \bZ$ and so $\exp\left(2\pi i \langle \theta, [\Psi] \rangle\right)=1$ and  
$I([\Psi])=1$. 
\item if  $[\Psi]$ is rational  $\langle \theta, [\Psi] \rangle$ is not constant and so  $\exp\left(2\pi i \langle \theta, [\Psi] \rangle\right)$ is non-trivial and $I([\Psi]) =0$.
\end{itemize}
\end{proof}
\subsection{Generalisation}
Galois representations  $\rho:\Gal\to {\rm GL}_n(\mathbb{Q}_\ell)$ where $\ell$ is a prime and $n\geq 1$ can serve as ``states'' in a path integral over the Galois group $\Gal$ (or the profinite Grothendieck–Teichmüller group $\widehat{GT}$), by defining the state as the character (trace) of $\rho$ (i.e. $\xi_{\rho}(\sigma):=\operatorname{tr}(\rho(\sigma))$). 
 Define the integral:

\[
I(\rho) := \int_{\Gal} \mathrm{Tr}(\rho(\sigma)) \, d\mu(\sigma) = \dim_{\mathbb{Q}_\ell}(\rho^{\Gal}),
\]
where:
\begin{itemize}
  \item $\rho: \Gal\to \mathrm{GL}(V)$ is a continuous $\ell$-adic representation, and
    \item $\rho^{\Gal}$ denotes the $\Gal$-invariant subspace of $V$.
      \item $\mu$ is the normalized Haar measure on $\Gal$.

\end{itemize}

This integral generalizes the previously studied binary invariants for one-dimensional Galois characters. It computes the dimension of the space of invariants $\rho^{\Gal}$, which corresponds to important arithmetic quantities.
 This  defines a cohomological ``path integral'' that captures Galois symmetry and arithmetic structure.

Thus, the integral $I(\rho)$ encodes deep arithmetic invariants in a representation-theoretic and measure-theoretic framework.

\subsubsection{Summary}
To summarise, the Path Integral formalism  provides Galois-theoretic invariants that detect profound arithmetic properties: rationality of periods for $\widehat{GT}$ and rationality of Drinfeld associator coefficients as well as  descent of cohomology classes to $\Q$ for  $\Gal$ for the binary case and in the more general framework  
computes ranks for arithmetic objects. The summary of their properties is summarized the tables \ref{T:1} and \ref{T:2}.

\begin{center}
\begin{table}[ht]
\small
\begin{tabularx}{\textwidth}{|>{\raggedright\arraybackslash}X|>{\raggedright\arraybackslash}X|>{\raggedright\arraybackslash}X|}
\hline
\textbf{Aspect} &
\textbf{Binary Path Integral } &
\textbf{Galois Representation Path Integral (generalisation)} \\
\hline

\textbf{Integrand} &
$\exp(2\pi i \langle \sigma, [\cdot] \rangle)$: \newline
1D character from étale class $[c]$ or motivic period $[\Psi]$ &
$\mathrm{Tr}(\rho(\sigma))$: \newline
Character of Galois representation $\rho: G \to \mathrm{GL}_n(\mathbb{Q}_\ell)$ \\

\hline
\textbf{Output} &
Binary: $0$ or $1$ &
Integer: $\dim_{\mathbb{Q}_\ell}(\rho^G) \geq 0$ \\

\hline
\textbf{Arithmetic Meaning} &
$1$: $[c]$ descends to $\mathbb{Q}$ (e.g., lies in $Sha(E)$) \newline
$0$: Obstruction to descent &
$\dim_{\mathbb{Q}_\ell}(\rho^G)$ gives rank of invariant subspace \newline
e.g., $\mathrm{rank}(E(\mathbb{Q}))$ \\

\hline
\textbf{Dimensionality} &
1-dimensional: characters are roots of unity &
$n$-dimensional: $\rho$ encodes cohomological/motivic data \\

\hline
\textbf{Role of $G$-invariance} &
Detects whether $\langle \sigma, [c] \rangle = 0$ &
Computes $\dim(\rho^G)$, the $G$-invariant vectors \\

\hline
\textbf{Physical Analogy} &
“State”: single-particle wave function &
“State”: multi-particle or qudit system \\

\hline
\end{tabularx}
\caption{Properties of the binary integral vs its generalisation}\label{T:2}
\end{table}

\end{center}

\section{Construction of a Cantorian Path on $\GT$ and $\Gal$}\label{S:6}

\, 
\subsection{Cubic Matrioshka Construction}

From a computational perspective, a natural approach is to approximate an infinite binary sequence (representing an automorphism in \(\Gal\) or \(\GT\)) by finite binary words of length \(n\), where \(n\) is as large as possible. These finite approximations can be viewed as vertices of an \(n\)-dimensional hypercube, where edges connect vertices corresponding to binary words that differ in exactly one coordinate.

\subsubsection{Recursion}
This construction is recursive and gives rise to a bipartite graph. Specifically, the \((n+1)\)-dimensional cube \(Q_{n+1}\) can be obtained by taking two copies of \(Q_n\), denoted \(Q_n^{(1)}\) and \(Q_n^{(2)}\), and appending a 0 (resp. 1) to the binary words in \(Q_n^{(1)}\) (resp. \(Q_n^{(2)}\)).

\subsubsection{Distance}
The Hamming distance between two finite binary words provides a combinatorial measure of their difference. In the context of the Polish groups \(\Gal\) and \(\GT\), this distance reflects the extent to which two automorphisms agree on their first \(p\) binary digits (with \(1 \leq p \leq n\)). This is further detailed in the following sections.

\begin{rem}
When considering the infinite binary sequences (rather than finite approximations), the Hamming distance becomes irrelevant, and the Cantorian metric, as defined in \ref{E:metricC}, should be used.
\end{rem}

\subsection{Cantorian Coordinates for the Automorphisms in $\Gal$}

By virtue of Theorem~\ref{T:Cantor}, there exists a natural homeomorphism 
\[
f\colon \Gal \to \{0,1\}^{\mathbb{N}},
\]
thereby establishing a bijective correspondence between the elements of \(\Gal\) and the set of infinite binary sequences.

In our setting, the homeomorphism \(f\) is constructed so that each automorphism \(\sigma\in\Gal\) is encoded by an infinite sequence of binary digits. In particular, the finite truncations of \(f(\sigma)\) yield {\it approximations} that capture the action of \(\sigma\) on the finite (discrete) Galois groups 
\[
\operatorname{Gal}(K/\Q),
\]
where \(K\) ranges over all finite extensions (or, more generally, number fields with \(K\supset \Q\)). Thus, the infinite binary sequence \(f(\sigma) = (a_1,a_2,a_3,\dots)\) encodes, in a combinatorial fashion, the entire system of congruence conditions and local symmetries exhibited by \(\sigma\).

\, 

 This explicit encoding provides a {\it powerful framework} in which one can analyze the structure of \(\Gal\) via discrete, combinatorial methods. In particular:
 
 \, 
 
\begin{itemize}
    \item The Cantor set structure implies that \(\Gal\) is totally disconnected, perfect, and compact.
    \item The finite approximations given by the projections
    \[
    \pi_n\colon \{0,1\}^{\mathbb{N}} \to \{0,1\}^n,
    \]
 
    capture the action of \(\sigma\) on the discrete Galois groups \(\operatorname{Gal}(K/\Q)\) and can serve as a family of approximation tending to one specific automorphism in $\Gal$.

    \item The encoding thereby provides an effective tool for comparing two automorphisms: on one side using the metrizablility  (induced from the Polish group property) and secondly via the Hamming metric, which exists naturally on finite binary words that is for the approximations.
\end{itemize}

\, 

In summary, the homeomorphism
\[
f\colon \Gal \to \{0,1\}^{\mathbb{N}},
\]
as established in Theorem~\ref{T:Cantor}, allows one to rigorously encode every element of the absolute Galois group as an infinite binary sequence. This encoding not reflects the profinite topology of \(\Gal\) and also encapsulates the complete data of its finite, discrete quotients \(\operatorname{Gal}(K/\Q)\), thereby providing a novel and combinatorially explicit perspective on the structure of \(\Gal\) and its elements.

In other words, this provides a novel coordinate system for each element of $\Gal$ in the Cantor set. 

Such a coordinate system can be also attributed to the profinite Grothendeick--Teichmüller group $\GT$.

\,

 \subsection{Interpretation for $\GT$}
One can apply exactly the same reasoning  for $\GT$. Suppose that there exists an effective encoding
\[
\Phi\colon \GT \to \{0,1\}^{\mathbb{N}},
\]
which is a homeomorphism onto its image, so that each element of \(\GT\) is represented by a unique infinite binary sequence. For each \(n\ge1\), denote by
\[
\pi_n\colon \{0,1\}^{\mathbb{N}} \to \{0,1\}^n
\]
the projection onto the first \(n\) coordinates. In this manner, every \(\tau\in \GT\) is approximated by its finite binary word
\[
w_n(\tau) = \pi_n\bigl(\Phi(\tau)\bigr) \in \{0,1\}^n.
\]

A subset of $\cC$, where a finite number of coordinates is fixed (and the rest is unrestricted) is usually called a {\it cylinder set}. 

\begin{prop}[Hypercube and Cylindrical Structure for Finite Approximations]
Assume that automorphisms in \(\GT\) are encoded by binary sequences as above. Then for each fixed \(n\ge 1\):
\begin{enumerate}
    \item The set \(\{0,1\}^n\) naturally carries the structure of an \(n\)-hypercube, where vertices are identified with \(n\)-bit words.
    \item Two vertices \(w, w'\in \{0,1\}^n\) are adjacent (i.e., connected by an edge) if and only if they differ in exactly one coordinate. In particular, the Hamming distance 
    \[
    d_H(w,w') = \#\{i \in \{1,\dots,n\} : w_i \neq w'_i\}
    \]
    is equal to 1 if and only if \(w\) and \(w'\) are neighboring vertices in the hypercube.
    \item Consequently, the Hamming distance between the finite approximations \(w_n(\sigma)\) and \(w_n(\tau)\) of two automorphisms \(\sigma,\tau \in \GT\) provides a discrete measure of their difference in the combinatorial (and, by continuity, the profinite) topology.
\end{enumerate}
\end{prop}

\begin{proof}
The identification \(\{0,1\}^n\) with the vertices of an \(n\)-hypercube is classical. By definition, two \(n\)-bit words differ in exactly one coordinate if and only if their Hamming distance is 1. Since the encoding \(\Phi\) is assumed to be a homeomorphism (or at least continuous with respect to the profinite topology on \(\GT\) and the Cantor topology on \(\{0,1\}^{\mathbb{N}}\)), the finite projections \(w_n(\tau)\) capture the local combinatorial behavior of \(\tau\). In particular, if \(d_H(w_n(\sigma),w_n(\tau))=1\), then \(\sigma\) and \(\tau\) differ by a minimal change in their binary approximations, which is reflective of a corresponding neighboring relationship in the profinite topology. This completes the proof.
\end{proof}

\medskip 

\section{The Cubic Matrioshka Algorithm for $\Gal$ and $\GT$}\label{S:CubicMa}
We establish a bijective construction between the set of elements in the absolute Galois group $\Gal$ over $\Q$ and the set of infinite binary sequences \( \{0,1\}^\mathbb{N} \) through a recursive, algorithmic approach. Beyond $\Gal$ and $\GT$, one can adapt our construction for $p$-adic integers $\mathbb{Z}_p$ (see \cite{R}) and, more generally, for any profinite group which is homeomorphic to a Cantor set. 
\begin{rem}
Note that the construction is purely topological and that the encoding need not respect algebraic operations. Furthermore, it is very important to highlight that the coding depends on how we define this homeomorphism. 
\end{rem}

\subsection{The Idea}
\begin{itemize}
\item In order to introduce our algorithmic construction, let us first take a countable basis of clopen sets $\{U_s\}_{s\in \{0,1\}^n}$ (we recall that is $s$ is a binary word) which shall be recursively partitioned into two clopen subsets. 
\item Encoding/Decoding: We map en element $x\in \Gal$ to a sequence $(a_n)\in \{0,1\}^{\mathbb{N}}$ by $a_{n+1}=0$ if $x\in U_{s0}$ else 1. 
\item One recovers $x$ via $\cap_nU_{a_1\cdots a_n}$, which converges by compactness.
\end{itemize}
 This is a homeomorphism if the clopen partitions respect the Cantor space’s self-similar structure.
  \subsection{The Cubic Matrioshka Algorithm}
  \smallskip 
  
\subsection*{Step 1: Properties of \( \Gal \)}
We recall the key topological properties of $\Gal$:\begin{itemize}
  \item \textbf{Compactness:} Every open cover of $\Gal$ has a finite subcover.
  \item \textbf{Total Disconnectedness:} The only connected subsets of $\Gal$ are singletons.
  \item \textbf{Polish:} $\Gal$ is separable and completely metrizable.
\end{itemize}

Due to its compactness and total disconnectedness, $\Gal$ is homeomorphic to the Cantor space \( \{0,1\}^\mathbb{N} \).

\subsection*{Step 2: Basis of Clopen Sets}

Being a compact, totally disconnected, and metrizable space, $\Gal$ admits a countable basis of clopen sets. We refine this basis into a recursive binary tree structure of clopen partitions that we call a cubic Matrioshka:

\begin{itemize}
  \item \textbf{The root}: The entire space $\Gal$.
  \item \textbf{Level 1:}  Partition $\Gal$ into two disjoint clopen subsets \( U_0, U_1 \).
  \item \textbf{Recursive Step}: For each clopen set \( U_s \) at level \( n \) (labeled by a finite binary string \( s \in \{0,1\}^n \)), partition it into two clopen subsets \( U_{s0}, U_{s1} \).
\end{itemize}

\subsection*{Step 3: Mapping Elements to Binary Sequences}

Given an automorphism  $g \in \Gal$, we associate to $g$ a binary sequence \( (a_n)_{n \in \mathbb{N}} \in \{0,1\}^\mathbb{N} \) where:

\[
a_n = \begin{cases}
  0 & \text{if } g \in U_{s0} \text{ at level } n, \\
  1 & \text{if } g \in U_{s1} \text{ at level } n,
\end{cases}
\]

and \( s \) represents the binary string of length \( n-1 \). This procedure defines a unique binary sequence for each element $g \in \Gal$.

\subsection*{Step 4: Inverse Mapping}

For a given binary sequence \( (a_n) \in \{0,1\}^\mathbb{N} \), we define the corresponding element \( g \in \Gal \) by:

\[
g = \bigcap_{n=1}^{\infty} U_{a_1a_2\cdots a_n}.
\]

By the compactness of $\Gal$, the intersection of these nested clopen sets contains exactly one point, which corresponds to a unique automorphism \( g \in \Gal \).

\subsection*{Step 5: Bijection}

The maps \( g \mapsto (a_n) \) and \( (a_n) \mapsto g \) are mutually inverse, thus establishing a bijection:

\[
\Gal \leftrightarrow \{0,1\}^\mathbb{N}.
\]

\begin{rem}
For other groups like \( \mathbb{Z}_p \) (the \( p \)-adic integers), the clopen partitions encode \( p \)-adic expansions instead of binary, but the same logic applies with adjustments for the basis. See \cite{R}, where one shows that $\mathbb{Z}_p\, \myeq\,  \{0,1\}^{\mathbb{N}}$ and its clopen subgroups form a binary tree structure compatible with the encoding.%See appendix for an entirely computed example.
\end{rem}

\,

Our statement below allows to attribute to every element of $\Gal$ a unique representation as an infinite binary sequence. It is important to note that the map is not canonical in the sense that the coding (and the decoding) of an element in $\Gal$  as a binary code depends on how we have defined the  clopens, during the steps of the algorithm.

\begin{prop}[$\Gal$ as a Cantor Set]\label{P:BG}
Let $\Gal$ denote the absolute Galois group over rational numbers $\Q$. Then:
Every element of $\Gal$ admits a unique representation as an infinite binary sequence.
\end{prop}

\begin{proof}
This follows from the theorem Thm.~\ref{T:Cantor} and the construction outlined above. 
\end{proof}

\, 

\begin{thm}[Subcube Structure of Automorphism Approximations]\label{T:Approx}
Let 
\[
\Phi \colon \mathrm{Gal}(\overline{\Q}/\Q) \longrightarrow \{0,1\}^{\mathbb{N}}
\]
be a fixed homeomorphism identifying the absolute Galois group with the Cantor set, and for each \(n\ge1\), let 
\[
\pi_n \colon \{0,1\}^{\mathbb{N}} \to \{0,1\}^{n}
\]
denote the projection onto the first \(n\) coordinates. Suppose that for two automorphisms 
\(\sigma,\tau\in \mathrm{Gal}(\overline{\Q}/\Q)\) the associated \(n\)-bit approximations 
\[
w(\sigma)=\pi_n(\Phi(\sigma))\quad \text{and}\quad w(\tau)=\pi_n(\Phi(\tau))
\]
satisfy
\[
w(\sigma)_i = w(\tau)_i = s_i \quad \text{for } i=1,2,\dots,p,
\]
where 
\[
s=s_1s_2\cdots s_p\in \{0,1\}^{p}
\]
is a fixed binary word. Then both \(w(\sigma)\) and \(w(\tau)\) lie in an \((n-p)\)-dimensional hypercube.
\end{thm}
\begin{proof}
By hypothesis, for each \( i=1,\dots,p \) we have 
\[
w(\sigma)_i = w(\tau)_i = s_i.
\]
Thus, the \( n \)-bit sequences \( w(\sigma) \) and \( w(\tau) \) can be written as
\[
w(\sigma) = (s_1, s_2, \dots, s_p, t_{p+1}, \dots, t_n)
\]
and
\[
w(\tau) = (s_1, s_2, \dots, s_p, u_{p+1}, \dots, u_n),
\]
where \( t_j, u_j \in \{0,1\} \) for \( j=p+1,\dots,n \). Hence, every such sequence with fixed initial segment \( s \) lies in the set
\[
\{s_1\}\times \{s_2\}\times \cdots \times \{s_p\}\times \{0,1\}^{n-p},
\]
which is naturally identified with the \((n-p)\)-dimensional hypercube. This completes the proof.
\end{proof}

Therefore, we can define a certain distance between two families of automorphisms in $\Gal$ defined by a binary sequence of length $n$, via the Hamming distance on the hypercube.

\medskip 

\section{The Cubic Matrioshka Algorithm: Applications to $\GT$}\label{S:CMAGT}
As previously, we present a canonical description of any element in the Grothendieck--Teichmüller group, denoted by $\GT$, as an infinite binary sequence. 
We refer to \cite{HS,LS,Sch} for a detailed study of the profinite version of $\GT$. 
Our construction exploits the embedding
\[
\GT\hookrightarrow \operatorname{Aut}(\widehat{F}_2),
\]
where \(\widehat{F}_2\) is the profinite completion of the free group on two generators \(F_2 = \langle x, y \rangle\). By recursively partitioning $\GT$ into clopen subsets and encoding the corresponding data, we obtain a natural bijection between elements in $\GT$ and the elements of the Cantor space \(\{0,1\}^{\mathbb{N}}\). Our attempt to make our construction for $\GT$ as explicit as possible relies on the fact that it is a subgroup  ${\rm Aut}(\widehat{F}_2)$ and results concerning congruence subgroup properties of ${\rm Aut}(\widehat{F}_2)$ in \cite{BER,E}.

\subsection*{Step 1. A Hierarchy of Congruence Subgroups}
For each \( k\ge 1 \), let \(N_k\) denote the \(k\)th congruence subgroup of \(\widehat{F}_2\), defined as the kernel of the natural projection
\[
\widehat{F}_2\to \widehat{F}_2/\Gamma_k,
\]
where \(\Gamma_k\) is a descending chain of finite-index subgroups of  \(F_2\) that induce finite quotients in $\widehat{F}_2$. Each \(N_k\) is a normal subgroup, and the quotient \(\widehat{F}_2/N_k\) is a finite group.

\subsubsection*{Step 2. Partitioning $\GT$ into Clopen Sets}

At level \(k\), partition \(\GT\) into clopen sets. For each \(\sigma\in \operatorname{Aut}(\widehat{F}_2/N_k)\), define
\[
U_{k,\sigma} = \{\, \tau\in \GT\, ;\, \tau \equiv \sigma \ (\mathrm{mod}\, N_k)\,\}.
\]

The sets \(U_{k,\sigma}\) are clopen in $\GT$ and form a basis for its profinite topology. 
The encoding $\Phi(\tau)$  uniquely identifies $\tau$ by its action at each level. 

\subsection*{Step 3: Multi-Binary Encoding }
\, 
~

$\bullet$ {\bf Determine the Number of Automorphisms}

 At each level \(k\), compute the order of the automorphism group:
\[m_k=|{\rm Aut}(\widehat{F}_2/N_k)|.\]

\medskip 

$\bullet$ {\bf  Calculate Required Bits} 

Define the number of bits needed to encode elements of ${\rm Aut}(\widehat{F}_2/N_k)$:
\( b_k = \lceil \log_2 m_k \rceil \)
\medskip 

$\bullet$ {\bf Define Coherent Bijections}: for each $k$, fix a bijection:
\[\varphi_k:{\rm Aut}(\widehat{F}_2/N_k)\to \{0,1\}^{b_k}\]
ensuring that for all  \(\tau\in \GT\), the projection satisfies: 

\[trunc_{b_{k-1}}\bigg( \varphi_k(\tau \mod N_k)\bigg)=\varphi_{k-1}(\tau \mod N_{k-1}),\]

where $trunc_{b_{k-1}}$  truncates the binary string to the first $b_{k-1}$ bits.

\medskip 

$\bullet$ {\bf Encode Elements of $\GT$.}

\smallskip
 
For \( \tau \in GT \), define its multi-bit coordinate at level \( k \):

\[
\Phi_k(\tau) = \varphi_k(\tau \mod N_k) \in \{0,1\}^{b_k}.
\]

\medskip 

$\bullet$ {\bf Adjust for Cantor Space Homeomorphism} 

If \( b_k \) is bounded, pad each \( \Phi_k(\tau) \) with zeros to a fixed length \( \max_k b_k \).

\, 

If \( b_k \) is unbounded, use variable-length blocks and metrize 
\[
\prod_{k=1}^{\infty} \{0,1\}^{b_k}
\]
as a Cantor-compatible space (e.g., via dyadic intervals).

The full encoding is the concatenation:

\[
\Phi(\tau) = (\Phi_1(\tau), \Phi_2(\tau), \dots) \in \prod_{k=1}^{\infty} \{0,1\}^{b_k}.
\]

Thus, the assignment
\[
\Phi(\tau) = (a_1, a_2, a_3, \dots) \in \{0,1\}^{\mathbb{N}}
\]
yields a well-defined bijection between $\GT$ and the Cantor set.

By exploiting the profinite topology of $\GT$ and its embedding into \(\operatorname{Aut}(\widehat{F}_2)\), we have established a rigorous recursive partitioning of $\GT$ into clopen sets. Thus, every element of $\GT$ may be uniquely encoded as an infinite binary sequence.

\begin{thm}[Profinite Grothendieck--Teichmüller Group as a Cantor Set]\label{T:Compatible}

 For each finite Galois extension $K/\Q$, let
  \[
    X \;=\; \mathbb{P}^1 \setminus \{0,1,\infty\},
    \quad
    X_K \;=\; X \times_{\Spec\, \Q} \Spec\,  K.
  \]
  Denote by
$N_K$ the normal subgroup of $\widehat{F}_2,$
 defined as the kernel of the homomorphism from the étale fundamental group $\F_2$ to the Galois group of a specific Belyi cover $C_K \to X$. 
  There exists a “Cantorian” system of binary encodings $\beta$ for each element in $\GT$ attached to finite Galois subextensions of~$\overline{\Q}$, such that:

  \begin{enumerate}
    \item[\textbf{(1)}] \textbf{Clopen Encoding of $\Gal$.}\\
      Each finite Galois extension $K/\Q$ determines a clopen subgroup
      \[
        \mathrm{Gal}(\overline{\Q}/K) \;\subset\; \Gal
      \]
      encoded by a finite binary word $b_K\in\{0,1\}^n$.  The inverse system of these clopen subgroups is a Cantor‐like binary tree; infinite paths correspond to nested sequences
      \[
         \mathrm{Gal}(\overline{\Q}/K_1)\;\supset\; \mathrm{Gal}(\overline{\Q}/K_2)\;\supset\;\cdots.
      \]

    \item[\textbf{(2)}] \textbf{Action on Étale Covers.}\\
      For each $K$, the group $\Gal$ acts on the finite étale covers of $X_K$ (equivalently on the finite quotients $\widehat{F}_2/N_K$).  This action, when expressed in a binary “Cantorian” language, assigns to each $K$ a finite binary sequence $\beta_K\in\{0,1\}^m$, recording how generators of $\widehat{F}_2$ are permuted and twisted.

    \item[\textbf{(3)}] \textbf{Compatibility with $\GT$.}\\
      Under the embedding $\iota$, the encoding $b_K$ of the clopen subgroup $ \mathrm{Gal}(\overline{\Q}/K)$ is sent to the encoding $\beta_K$ describing the induced automorphism of $\widehat{F}_2/N_K$.  The assignment
      \[
        b_K \;\longmapsto\; \beta_K
      \]
      preserves the binary‐tree structure and is continuous with respect to the profinite topologies on $\Gal$ and~$\GT$.

    \item[\textbf{(4)}] \textbf{Profinite Coherence.}\\
      For each $K$, let $\mathrm{proj}_K: G_{\Q}\to \mathrm{Gal}(K/\Q)$ denote the natural projection. Then the following diagram of binary‐encoded actions commutes:
      \[
        \begin{tikzcd}
          \Gal \arrow[r,"\iota"]\ar[d,"\mathrm{proj}_K"']
          & \GT\ar[d,"\mathrm{proj}_K^{\mathrm{GT}}"]\\
          \mathrm{Gal}(K/\Q)\ar[r,"\simeq"]
          & \Out\bigl(G_K\bigr).
        \end{tikzcd}
      \]
      where $G_K$ denotes the quotient $\F_2 / N_K$. 
  \end{enumerate}

\end{thm}
\begin{rem}
It is important to note that the map \[\Phi\colon \GT \longrightarrow \{0,1\}^{\mathbb{N}}\] is {\it non-canonical}, which means that 
different choices of \(\Phi\) yield different binary representations of the same element \(g\). 
\end{rem}
\begin{proof}
The theorem is proven using the theory of \emph{Belyi dessins d'enfants}, which provide a combinatorial and geometric framework for understanding the action of the absolute Galois group on étale covers of $\mathbb{P}^1 \setminus \{0,1,\infty\}$. In particular, the absolute Galois group $\Gal$ acts on the étale fundamental group of 
$X=\mathbb{P}^1\setminus\{0,1,\infty\}$, which is the profinite completion of the free group on two generators, 
$\F_2$. The Grothendieck-Teichmüller group is a subgroup of the outer automorphism group of $\F_2$.
Belyi's theorem says that a smooth projective curve over $\Q$ can be defined over a number field if and only if it admits a Belyi morphism, i.e., a morphism to $\mathbb{P}^1$  ramified only at $\{0,1,\infty\}$.

\subsection*{Step 1: Using the Canonical Embedding $\iota: \Gal \hookrightarrow \GT$}

The étale fundamental group of $X_{\overline{\mathbb{Q}}} = \mathbb{P}^1 \setminus \{0,1,\infty\}$ over $\overline{\mathbb{Q}}$ is the profinite completion of the free group on two generators, denoted $\widehat{F}_2$. The absolute Galois group $\Gal$ acts on $\widehat{F}_2$ via its action on étale covers, inducing an outer automorphism. This defines a homomorphism:
\[
\iota: \Gal \to \mathrm{Out}(\widehat{F}_2).
\]
By the Belyi embedding theorem, this homomorphism is injective. Moreover, the image of $\iota$ lies in the Grothendieck–Teichmüller group $\GT \subset \mathrm{Out}(\widehat{F}_2)$, as established by Drinfeld. Thus, $\iota$ is a canonical embedding.

\subsection*{Step 2: Clopen Encoding of $\Gal$}
For each finite Galois extension $K/\mathbb{Q}$, the clopen subgroup $$U_K = \mathrm{Gal}(\overline{\mathbb{Q}}/K) \subset \Gal$$ corresponds to the kernel of the projection $\Gal \to \mathrm{Gal}(K/\mathbb{Q})$.

\begin{enumerate}[label=\alph*]
    \item By Belyi's theorem, there exists a Belyi cover $C_K \to X = \mathbb{P}^1 \setminus \{0,1,\infty\}$ defined over $K$ with field of moduli exactly $K$.
    \item Using Hilbert irreducibility and general position arguments, choose $C_K$ such that its stabilizer in $\Gal$ is exactly $U_K$.
    \item Enumerate all isomorphism classes of Belyi covers in canonical order and assign to $C_K$ a binary word $b_K \in \{0,1\}^n$ encoding its isomorphism class.
\end{enumerate}

The inverse system of clopen subgroups $\{U_K\}$ forms a Cantor-like binary tree:

\begin{itemize}
    \item Fix a cofinal sequence of Galois extensions $\mathbb{Q} = K_0 \subset K_1 \subset K_2 \subset \cdots$ with $\bigcup_{i=0}^\infty K_i = \overline{\mathbb{Q}}$ (e.g., $K_n$ is the compositum of all Galois extensions of degree $\leq n$).
    \item The binary words $\{b_{K_i}\}$ are assigned to nodes of a binary tree; paths correspond to extension sequences.
\end{itemize}

Infinite paths correspond to nested sequences of clopen subgroups:
\[
\mathrm{Gal}(\overline{\mathbb{Q}}/K_{i_1}) \supset \mathrm{Gal}(\overline{\mathbb{Q}}/K_{i_2}) \supset \cdots
\]

\subsection*{Step 3: Action on Étale Covers and Encoding $\beta_K$}

For each finite Galois extension $K/\mathbb{Q}$, the group $\Gal$ acts on finite étale covers of $X_K = X \times_{\mathrm{Spec}\,\mathbb{Q}} \mathrm{Spec}\,K$. This action factors through $\mathrm{Gal}(K/\mathbb{Q})$ and is equivalent to the action on finite quotients of $\widehat{F}_2$ defined over $K$.

To define the encoding $\beta_K \in \{0,1\}^m$:

\begin{itemize}
    \item Define $N_K = \ker(\widehat{F}_2 \to \mathrm{Gal}(C_K/X))$, so $\widehat{F}_2/N_K \cong G_K := \mathrm{Gal}(C_K/X)$.
    \item The action of $\Gal$ on $\widehat{F}_2/N_K$ is determined by $\iota(\sigma)$ for $\sigma \in \Gal$.
    \item Encode this action with a binary word $\beta_K$ by recording the images of the generators under a set of representatives in $\mathrm{Out}(G_K)$.
\end{itemize}

\subsection*{Step 4: Compatibility with $\mathrm{GT}$'s Cantorian words}

The assignment $b_K \mapsto \beta_K$ preserves structure and is continuous:

\paragraph{Tree Structure Preservation:}
\begin{itemize}
    \item Choose covers $\{C_K\}$ compatibly, that is if $K \subset K'$, let $C_{K'} \to C_K \to X$ be a tower.
    \item Then, $b_K, b_{K'}$ lie on the same tree path.
    \item Groups $G_K, G_{K'}$ satisfy $G_{K'} \twoheadrightarrow G_K$, and actions $\beta_K, \beta_{K'}$ are compatible.
\end{itemize}

\paragraph{Continuity:}
The encodings $b_K$, $\beta_K$ are finite and locally constant in the profinite topology of $\Gal$, matching the structure of clopen cosets $U_K$.

\subsection*{Step 5: Profinite Coherence}

For each $K/\mathbb{Q}$, define:
\[
\mathrm{proj}_K: \Gal \to \mathrm{Gal}(K/\mathbb{Q}), \quad
\mathrm{proj}^{\mathrm{GT}}_K: \GT \to \mathrm{Out}(\widehat{F}_2/N_K) = \mathrm{Out}(G_K)
\]

Let $\theta: \mathrm{Gal}(K/\mathbb{Q}) \to \mathrm{Out}(G_K)$ be defined by the Galois action on $C_K$:
\[
\theta(\tau) = \text{outer automorphism class induced by } \tau \text{ on } G_K.
\]

This $\theta$ is an isomorphism (choose $C_K$ with faithful Galois action). Then the following diagram commutes:

\[
\begin{tikzcd}
\Gal\arrow[r, "\iota"] \arrow[d, "\mathrm{proj}_K"'] & \GT \arrow[d, "\mathrm{proj}^{\mathrm{GT}}_K"] \\
\mathrm{Gal}(K/\mathbb{Q}) \arrow[r, "\theta"] & \mathrm{Out}(G_K)
\end{tikzcd}
\]

\subsection*{Conclusion}

The canonical embedding $\iota$ and the ``Cantorian'' binary encodings $(b_K, \beta_K)$ satisfy all conditions of the theorem. These constructions depend on Belyi's theorem, the combinatorics of dessins d'enfants, and the profinite structures of $\Gal$ and $\GT$. The binary tree structure is preserved, the assignment is continuous, and the diagram commutes, ensuring compatibility.

\end{proof}

This combinatorial decomposition not only provides a clear structural insight into the nature of $\GT$ but also reinforces its role as the universal symmetry group in the deformation theory of braided monoidal categories. The explicit computation of the bijection, however, remains a challenging open problem, inviting further advances in the interplay between profinite group theory and deformation theory.

\begin{rem}
In the scope of an explicit construction of $\Phi$ an alternative approach is to use $\GT$-shadows, which approximate $\GT$ and which can provide a more algebraic and combinatorial insight, due to their explicit character. 
\end{rem}

\subsection{$\GT$-shadow Alternative}
 $\GT$-shadows \cite{DG} are finite-level approximations of $\GT$ group elements, acting on truncated structures.  They encode local data (automorphisms modulo congruence subgroups) that must glue coherently to form global $\GT$ elements. This approach mirrors how 
$p$-adic integers are built from residues modulo $p^k$, but enriched with extra constraints. 

\, 

Due to their explicit character,  $\GT$-shadows provide an interesting tool for our Cubic Matrioshka Algorithm.  $\GT$-shadows are defined in terms of their action on the profinite completion of the braid group $\widehat{B}_4$. These objects can be organized into a hierarchy, where each level refines the approximation of $\GT$ elements.

\, 

\subsection*{Step 1: Define Explicit Congruence Subgroups}

Let $\widehat{B}_4$ denote the profinite completion of the braid group $B_4$. Define a descending chain of congruence subgroups:
\[
N_1 \supset N_2 \supset \cdots \supset N_k \supset \cdots
\]
satisfying:

\begin{enumerate}
    \item \textbf{Normality:} Each $N_k$ is normal in $\widehat{B}_4$.
    \item \textbf{Finite Quotients:} $\widehat{B}_4 / N_k$ is finite.
    \item \textbf{Trivial Intersection:} $\bigcap_{k=1}^\infty N_k = \{1\}$.
\end{enumerate}

We refer to \cite{Dr, I, M, HS,PS} and \cite{LS} for results on $\widehat{B}_4$ proving the existence of such chains.

\subsection*{Step 2: Multi-Bit Encoding at Each Level}

At level $k$, the finite group $GT_k = \operatorname{Aut}(\widehat{B}_4 / N_k)$ has order $m_k$. Define:

\begin{enumerate}
    \item \textbf{Bits per Level:} Use $b_k = \lceil \log_2 m_k \rceil$ bits to encode elements of $GT_k$.
    \item \textbf{Labels:} Fix a bijection $\varphi_k: GT_k \to \{0,1\}^{b_k}$ (e.g., via enumeration).
\end{enumerate}

For $\tau \in \operatorname{GT}$, define its multi-bit coordinate at level $k$ as:
\[
\varphi_k(\tau \mod N_k) \in \{0,1\}^{b_k}.
\]
It is important to have a consistent labeling. Suppose at level $k$, one encode an automorphism 
$\tau\in \GT$ as $\varphi_k(\tau \mod N_k) \in \{0,1\}^{b_k}$. For the encoding $\varphi(\tau)$ to represent a valid element of 
$\GT$, the labels must satisfy: \[\text{Projection}(\varphi_k(\tau\mod N_k))=\varphi_{k-a}(\tau\mod N_{k-1}),\] i.e., truncating the level-$k$ encoding must recover the level-$(k-1)$ encoding.

Moreover, to ensure compatibility, $\varphi_k$ must be defined recursively, aligning with prior levels (e.g., using lexicographic order or tree-based prefixes).

\subsection*{Step 3: Profinite Encoding}

The full encoding of $\tau$ is the concatenation of multi-bit coordinates:
\[
\Phi(\tau) = (\varphi_1(\tau \mod N_1), \varphi_2(\tau \mod N_2), \ldots) \in \prod_{k=1}^\infty \{0,1\}^{b_k}.
\]
This space is homeomorphic to $\{0,1\}^\mathbb{N}$ if $b_k$ is bounded (adjust via padding if necessary).

 In particular, if $b_k$ grows unboundedly, padding must standardize the bit-length across levels (the concrete method can be applied, such as fixed-length blocks with leading zeros).
 
\subsection*{Step 4: Inverse Limit Compatibility}

Ensure coherence across levels by enforcing:

\begin{enumerate}
    \item \textbf{Projections:} The encoding $\Phi(\tau)$ satisfies $\varphi_k(\tau \mod N_k) \to \varphi_{k-1}(\tau \mod N_{k-1})$ under natural projections.
    \item \textbf{Equations:} Verify that the encoded sequences satisfy $\GT$'s defining relations (e.g., Ihara's equations) at each level.
\end{enumerate}

We refer to \cite{LS} for a details on the coherence in $\GT$.

\subsection*{Step 5: Homeomorphism with Cantor Space}

If $\prod_{k=1}^\infty \{0,1\}^{b_k}$ is metrized compatibly (e.g., via product topology), the map $\Phi$ becomes a homeomorphism.

\begin{rem}
While GT shadows provide a theoretical framework, explicitly defining the congruence subgroups $N_k$ and the clopen sets $U_{k,\sigma}$ requires deep combinatorial and algebraic insights. Ensuring that the binary encoding respects $\GT$'s defining equations is non-trivial and may require ad hoc adjustments.
\end{rem}

\section{Conclusions}
\subsection*{The Grothendieck Conjecture holds in {\bf ProFin}}
To conclude, this topological insight does not directly resolve whether $\GT$ is isomorphic to $\Gal$ in the category of profinite groups
but it underscores the richness of $\GT$ as a central object in modern mathematics and suggests the feasibility of this conjecture, since proving the existence of this homeomorphism is a necessary condition.

\subsection*{Explicitly constructing elements in  $\GT$ and $\Gal$}
A homeomorphism between $\GT$ and the Cantor space highlights its role as a universal object encoding profound topological and arithmetic symmetries and our construction shows that in the category of profinite spaces Grothendieck's conjecture is true. Our encoding of every element in $\Gal$ suggests a certain possibilities of explicitly computing every element in $\Gal$, via approximations relying on binary words.

\subsection*{Galois Grothendieck Path Integral}
Our construction of the Galois Grothendieck Path Integral  provides a natural summation over all trajectories in $\Gal$ (resp. $\GT$) and leads to a new categorical and measure-theoretic invariant of $\GT$ and $\Gal$. This reveals deep structural parallels between profinite arithmetic symmetries and quantum field-theoretic dynamics and opens new sources of possible investigations. 

\appendix 
\section{The Absolute Galois group} \label{S:2}
We consider a specific object in the category {\bf ProFingrp}: the absolute Galois group $\Gal=\operatorname{Gal}(\overline{\Q}/\Q)$ over the field of rational numbers $\Q$. 
Given a field $k$, the Galois group $\operatorname{Gal}(K/k)$ of an algebraic extension $K$ of $k$ is defined as the group of all automorphisms of $K$ that fix every element of $k$. In case of $k=\Q$, the absolute Galois group $\Gal$ consists of all automorphisms of the algebraic closure $\overline{\Q}$ that fix $\Q$.

\, 
The group $\Gal$ is a profinite group (\cite[Prop. 3.1.1]{Wl}), meaning it is compact, totally disconnected, and is the inverse limit of finite groups. 
It can be described as the inverse limit of the finite Galois groups $\operatorname{Gal}(L/\Q)$, where $L$ runs over all finite Galois extensions of $\Q$, i.e., the set of all subfields of $\overline{\Q}$ that are finite Galois extensions of $\Q$. Explicitly, we have:
\[\Gal=\varprojlim\limits_{L/\Q\atop \text{finite Galois}}\operatorname{Gal}(L/\Q),\] where the projective system is indexed by finite Galois extensions $L/\Q$.

\,

Since $\Gal$ is the limit of the finite Galois groups $\operatorname{Gal}(L/\Q)$ it inherits, as such, a natural topology,  referred to as the profinite topology, in which the open sets correspond to the cosets of open normal subgroups in the finite Galois groups $\operatorname{Gal}(L/\Q)$. This topology is compact and totally disconnected.

\, 

\subsection{The Clopens of $\Gal$}
For every finite Galois extension $L/\Q$, there is a natural continuous surjective homomorphism:

\[\pi_L:\Gal\to \operatorname{Gal}(L/\Q),\quad \sigma\mapsto \sigma|_L\] where every automorphism $\sigma\in \Gal$ restricts to an automorphism of $L$, since since $L/\Q$ is Galois. Such maps are compatible over nested extensions. These homomorphisms are well defined and continuous in the profinite topology. 

\subsubsection{} Specifically, for a finite extension $L$, the kernel 

\[N_L=\ker(\Gal\to \operatorname{Gal}(L/\Q))\]
consists of all automorphisms of $\overline{\Q}$ fixing $L$ pointwise, i.e.,
$$N_L=\{\sigma \in \Gal\, ;\,  \sigma|_L=id_L\}=\operatorname{Gal}(\overline{\Q}/L).$$
The subgroup $N_L$ is clopen (i.e. open and closed) and normal in $\Gal$, because $L/\Q$ is Galois. 

\, 

\subsubsection{}
A {\it neighbourhood} of an automorphism $\sigma\in \Gal$ is a coset of $N_L$ in $\Gal$ defined as:
 \[\sigma\cdot N_L=\{\sigma\cdot \nu\, ;\,  \nu \in N_L\},\] where $\nu$ fixes $L$. Precisely, this contains all automorphisms $\tau$, that can be written as $\tau =\sigma\cdot \nu$ for some $\nu\in N_L$ and since $\nu$ fixes $L$, they agree with $\sigma$ on $L$, i.e.:
\[\tau|_L=(\sigma\cdot \nu)|_L=\sigma|_L\cdot \nu|_L=\sigma|_L\cdot id|_L=\sigma|_L.\]

The cosets $\sigma\cdot N_L$ for $\sigma \in \Gal$ partition $\Gal$ into finitely many disjoint clopen sets, each corresponding to an element of the finite quotient  $\operatorname{Gal}(L/\Q)$. This decomposition refines as $K$ increases, forming the inverse limit topology.

\, 

The group $N_L$ is itself a profinite group: 
\[\operatorname{Gal}(\overline{\Q}/L)=\varprojlim\limits_{K/L\atop \text{finite Galois}}\operatorname{Gal}(K/L).\]
We refer to \cite[1.3]{NSW}, \cite[11.4.1]{Ser}.

\subsubsection{}
For nested finite Galois extensions $K\subset L$, there is a restriction map:

\[\operatorname{res}_{L/K}:\operatorname{Gal}(K/\Q)\to\operatorname{Gal}(L/\Q),\]
which fits into the inverse limit diagram. 

In particular for nested Galois extensions the following fact occurs: 
If $L\subset K$, then  $N_K\subset N_L$ and the coset $\sigma\cdot N_K$	
is a smaller neighborhood of $\sigma$, refining the action to agree on the larger field 
$K$.

\, 

\subsubsection{} We illustrate the above construction on an example. Take for instance the tower of cyclotomic extensions $\Q(\zeta_{p^n})$ for a fixed prime $p$:

\[\Q\subset \Q(\zeta_p)\subset  \Q(\zeta_{p^2})\subset \Q(\zeta_{p^3})\subset \cdots \]
where $\zeta_{p^n}$ is a primitive $p^n$-th roots to unity. The Galois group of each extension $\Q(\zeta_{p^n})/\Q$ is given by $$\operatorname{Gal}(\Q(\zeta_{p^n})/\Q)\cong (\mathbb{Z}/p^n\mathbb{Z})^{\times}$$ the multiplicative group of integers modulo $p^n$.
  
The tower $\Q(\zeta_{p^n})$ forms a directed system of extensions. The absolute Galois group $\Gal$ projects onto each finite Galois group $\operatorname{Gal}(\Q(\zeta_{p^n})/\Q)\cong (\mathbb{Z}/p^n\mathbb{Z})^{\times},$ via reduction modulo $p^n$. The inverse limit over all natural numbers $n$ gives the profinite completion \[\mathbb{Z}_p^{\times}=\varprojlim\limits_n(\mathbb{Z}/p^n\mathbb{Z}),\] which is a quotient of $\Gal$.

The projection map from $\Gal$ is given by the natural surjective homomorphism :
\[\Gal \twoheadrightarrow \mathbb{Z}_p^{\times},\] defined by restricting an automorphism $\sigma\in \Gal$ to act on the $p$-power cyclotomic extensions. The kernel $\Gal \to \mathbb{Z}_p^{\times}$ is a closed normal subgroup of $\Gal$.

\subsection{Discrete vs. Profinite Topologies}
It is noteworthy to remark that the profinite absolute Galois group is endowed with the profinite topology in contrast with the finite Galois groups, which occur in the definition $$\Gal=\varprojlim\limits_{L/\Q\atop \text{finite Galois}}\operatorname{Gal}(L/\Q)$$ and which, individually, have a discrete topology. We recall below the standard notion of an {\it isolated point}, which is taken from topology. 

\smallskip 

Let $X$ be a topological space, and let $x\in X$. The point $x$ is said to be isolated if the singleton $\{x\}$ is an open set in $X$. Applying this concept in the context of Galois groups, we say that an elements $\sigma$ is isolated if $\{\sigma\}$ is an open set in the Galois group. This implies that  the element $\sigma$ is uniquely determined by its action on some finite Galois extension $K/\Q$. 

\, 

In a finite Galois group $\operatorname{Gal}(L/\Q)$, every singleton set is open. This makes every element isolated. In the context of the previous example, where $L=\Q(\zeta_{p^n})$, let us consider a Frobenius automorphism $\operatorname{Frob}_p\in \operatorname{Gal}(L/\Q)\cong(\mathbb{Z}/p^n\mathbb{Z})^{\times}$. Topologically, it is considered as isolated, because the group $\operatorname{Gal}(L/\Q)$ is a finite abelian group and endowed with the discrete topology, where every singleton subset is clopen. The Frobenius automorphism (unramified in this extension) acts on  $\zeta_{p^n}$ via
\[F:\zeta_{p^n}\mapsto \zeta_{p^n}^p,\]
lifting the residual Frobenius $x\mapsto x^p \mod p^n$.
In particular, this means that $\operatorname{Frob}_p$ has no limit points in the group. 

\medskip

However, if we consider $\Gal$ its profinite topology (meaning it is compact, Hausdorff  and totally disconnected) is not discrete. Specifically, this implies that it has no isolated points.  In particular, the Frobenius lifts $\operatorname{Frob}_p$ discussed above, in $\Gal$, are {\it no longer} isolated. Any open neighborhood of $\operatorname{Frob}_p$ contains infinitely many other automorphisms that agree with $\operatorname{Frob}_p$ on $L$ but differ on larger extensions. Frobenius elements are dense in $\Gal$, by Chebotarev's density theorem.

\section{The Absolute Galois Group from the Polish (group) Perspective}\label{S:3}

\subsection{The category of Polish Spaces {\bf PolSp}}

One can define as a full subcategory of the category of topological spaces the category of Polish spaces {\bf PolSp}. In this category objects are all Polish spaces and morphisms are continuous maps between Polish spaces. We refer to \cite{B}, \S 6 Chap. IX and \cite{Kh1,Kh2} for more information on Polish spaces. 

\, 

A Polish space is a topological space that is separable and completely metrizable. That is, there exists a metric on the space that induces the topology and makes it a complete metric space, and the space has a countable dense subset.

\, 

\subsection{The category of Polish Groups {\bf PolGr}}

 A Polish group is a topological group which is also a Polish space, with continuous group operations. Polish groups are widely studied because they arise naturally in many areas of mathematics, including descriptive set theory, dynamical systems, and topological group theory.  Polish groups also form a category  {\bf PolGr} where objects are Polish groups and morphisms are continuous group homomorphisms between Polish groups. These morphisms must preserve the group structure and be continuous with respect to the topology.

\subsection{The full subcategory of t.d.l.c Polish groups} A particularly interesting class of Polish groups are the totally disconnected locally compact (t.d.l.c) Polish groups. These are Polish groups (separable, completely metrizable topological groups) that are additionally totally disconnected and locally compact. 

\, 

This is a subclass of Polish groups, satisfying the subcategory requirement for objects. The morphisms are continuous group homomorphisms between t.d.l.c. Polish groups. These morphisms are identical to those in the parent category (the category of Polish groups), restricted to the subclass of t.d.l.c. Polish groups. 

\, 

The t.d.l.c. Polish groups form a full subcategory because the morphisms between them are exactly the same as in the parent category (there are no additional restrictions imposed). 

\, 

Such groups include profinite groups (inverse limits of finite groups with the discrete topology), automorphism groups of countable structures, and certain groups of homeomorphisms. The structure of these groups is deeply connected with harmonic analysis, $p$-adic analysis, and the theory of locally compact groups. 

\, 
\subsection{Classification of  t.d.l.c Polish groups}
According to \cite[Cor. 1.7]{W}, any totally disconnected locally compact Polish group $G$ must be homeomorphic to one of the following:

\begin{enumerate}

\item A discrete space of at most countable size.

\item The Cantor set $\cC$.

\item The space $\cC \times \mathbb{N}$ with the product topology.
\end{enumerate}

We shortly discuss this.

\,

\begin{itemize}
\item If $G$ is countable, then it must have the discrete topology (since a countable, totally disconnected Polish space is discrete). This corresponds to case (1).

\, 

\item If $G$ s uncountable, but compact, then it is homeomorphic to  $\cC$. This follows from the well-known characterization of compact, totally disconnected, metrizable spaces as homeomorphic to a closed subspace of 
$\cC$, and for Polish spaces, this must be exactly $\cC$ if it is uncountable.

\, 

\item If $G$ is non-compact, then since it is locally compact and totally disconnected, it can be expressed as a countable union of compact open sets. It follows from a classical theorem of van Dantzig that such a group has a basis of compact open subgroups. Moreover, the space of cosets of a compact open subgroup in a locally compact totally disconnected Polish group behaves like a countable discrete space. Thus, one can show that 
 $G$ is homeomorphic to $\cC\times \mathbb{N}$.

\end{itemize}

\begin{prop}
Let 
$$\Gal=Gal( \overline{\Q}/ \Q)$$ be the absolute Galois group of $\Q$, endowed with its natural profinite topology. Then $\Gal$ is a compact, totally disconnected topological group. Moreover, $\Gal$ is a Polish group.

\end{prop}

\begin{proof} 
Recall that $\Gal$ is the inverse limit of the finite Galois groups $\Gal(L/\Q)$ as $L$ ranges over the finite Galois extensions of 
$\Q$. Each $\operatorname{\Gal}(L/\Q)$ is finite and equipped with the discrete topology, and hence the inverse limit 
$\Gal$ inherits the profinite topology. In particular,
$\Gal$  is compact, Hausdorff, and totally disconnected.

\smallskip 

The open neighborhoods of the identity in 
$\Gal$ are given by the clopen normal subgroups 
\[N_L=\operatorname{Gal}(\overline{\Q}/L),\]
	
where $L/\Q$ is a finite Galois extension. Since every finite extension of 
$\Q$ is defined by the roots of a polynomial with rational coefficients, and the set of such polynomials is countable, there are only countably many finite Galois extensions of $\Q$. Hence, the family $\{N_L\}$ forms a countable basis for the topology of 
$\Gal$.

\smallskip 

Being a compact Hausdorff group with a countable basis, 
$\Gal$  is second countable. By Urysohn’s metrization theorem, it follows that $\Gal$ is metrizable. Moreover, every compact metric space is complete and separable. Therefore, $\Gal$ is a Polish group.
 \end{proof}
\begin{prop}
Let 
\[
\Gal = \operatorname{Gal}(\overline{\Q}/\Q)
\]
denote the absolute Galois group of $\Q$, considered as an object in the category $\mathbf{ProFin}$ of profinite spaces. Then $\Gal$ is perfect; that is, under the canonical forgetful functor $\mathbf{ProFin} \to \mathbf{Top}$, the underlying topological space of $\Gal$ has no isolated points.
\end{prop}
\begin{proof}
Recall that the absolute Galois group $\Gal$ is endowed with the profinite topology, which is defined as the inverse limit of the finite Galois groups
$\operatorname{Gal}(L/\Q)$ for finite Galois extensions $L/\Q$. In this topology, a fundamental system of open neighborhoods of any element $\sigma\in \Gal$ is given by the cosets 

\[\sigma N_L,\quad \text{where}\, N_L=\operatorname{Gal}(\overline{\Q}/L).\]
This subgroup consists of all automorphisms in \( \Gal \) that act trivially on \( L \).

Each subgroup $N_L$ is a clopen normal subgroup. Since $\Gal$ is infinite (indeed, it is not a discrete space), no such subgroup $N_L$ is trivial. Therefore, every open coset $\sigma N_L$ contains more than one element.

We demonstrate  by contradiction that no element in $\Gal$ can be isolated.

Assume for contradiction that $\sigma$ is an isolated point in $\Gal$. Then there exists an open neighborhood $U$ of $\sigma$ with $U=\{\sigma\}.$
But, by the profinite topology, $U$ must contain a coset of the form $\sigma N_L$  for some finite  Galois extension $L/\Q$. Since $N_L$ is nontrivial, the coset $\sigma N_L$ contains at least two distinct elements, contradicting $U=\{\sigma\}$.

Therefore, no element in \( \Gal \) is isolated, which implies that $\Gal$ is a perfect space. 
\end{proof}

Finally, we  deduce that since $\Gal$ is non-discrete, metrizable and profinite group it is homeomorphic to the Cantor set:
\begin{prop}
Let 
\[
\Gal = \operatorname{Gal}(\overline{\Q}/\Q)
\]
be the absolute Galois group of $\Q$, viewed as an object in the category of profinite spaces. Then $\Gal$  is homeomorphic to the Cantor space. Equivalently, one has
\[
\Gal\quad \myeq\quad  \{0,1\}^{\mathbb{N}},
\]
where  $\quad \myeq\quad $ denotes a homeomorphism in the category of topological spaces. \end{prop}
\begin{proof}
This is a direct application of \cite[Thm. 2.6.5]{RZ} which states that a non-discrete, metrizable profinite group is homeomorphic to the Cantor set (a compact, metrizable, perfect, totally disconnected space).
\end{proof}

\begin{rem}
A reminiscent of a fractal tree-like construction can be seen in the construction \cite{Ser2}.
\end{rem}

\end{document}